\documentclass[11pt]{amsart}
\usepackage{amsmath,amsthm,amssymb,amscd} 
\usepackage{color}

\def\a{{\alpha}}
\def\g{\gamma}
\def\G{\Gamma}
\def\s{\sigma}
\def\l{\left}
\def\r{\right}
\def\k{\kappa}
\def\w{\omega}
\def\ep{\varepsilon}
\def\vp{\varphi}
\def\th{\theta}
\def\d{\delta}

\def\A{{\mathbb{A}}}
\def\R{{\mathbb{R}}}
\def\P{{\mathbb{P}}}
\def\Q{{\mathbb{Q}}}
\def\C{{\mathbb{C}}}
\def\Z{{\mathbb{Z}}}
\def\H{{\mathbb{H}}}
\def\F{{\mathbb{F}}}
\def\K{{\mathbb{K}}}

\def\SS{{\mathfrak{S}}}
\def\oo{\mathfrak{o}}

\def\O{{\mathcal{O}}}
\def\cA{{\mathcal{A}}}
\def\Sc{{\mathcal{S}}}

\def\m{{\mathbf{m}}}
\def\n{{\mathbf{n}}}

\def\GL{{\mathrm{GL}}}
\def\SL{{\mathrm{SL}}}
\def\GSp{{\mathrm{GSp}}}
\def\PGSp{{\mathrm{PGSp}}}
\def\Sp{{\mathrm{Sp}}}
\def\GSO{{\mathrm{GSO}}}
\def\GO{{\mathrm{GO}}}
\def\SO{{\mathrm{SO}}}
\def\O{{\mathrm{O}}}

\def\i{\infty}
\def\ir{\mathrm{i}}
\def\ol{\overline}
\def\t{\times}
\def\ot{\otimes}
\def\rt{\rtimes}
\def\bs{\backslash}
\def\wt{\widetilde}
\def\wh{\widehat}
\def\In{{\mathrm{Ind}}}
\def\Hom{{\mathrm{Hom}}}
\def\Ch{{\mathrm{Ch}}}
\def\dg{{\mathrm{diag}}}
\def\1{\mathbf{1}}

\def\bQ{\overline{\mathbb{Q}}}

\def\bk{\backslash}
\def\tha{{\tfrac{1}{2}}}
\def\ha{{\frac{1}{2}}}
\def\ds{\displaystyle}
\def\lra{\longrightarrow}

\numberwithin{equation}{section}

\newtheorem{thm}{Theorem}[section]
\theoremstyle{plain}
\newtheorem{Def}[thm]{Definition}
\newtheorem{prop}[thm]{Proposition}
\newtheorem{Cor}[thm]{Corollary}
\newtheorem{lem}[thm]{Lemma}

\theoremstyle{definition}
\newtheorem{Rem}[thm]{Remark}

\newcommand{\FRAC}[2]{\leavevmode\kern.1em \raise.5ex\hbox{\the\scriptfont0 #1}\kern-.1em/\kern-.15em\lower.25ex\hbox{\the\scriptfont0 #2}}
\title{\bf On some Siegel threefold related to the tangent cone of the Fermat quartic surface}
\author{Takeo Okazaki and Takuya Yamauchi}
\dedicatory{Dedicated to Professors van Geemen, Nygaard, and van Straten}
\date{}
\thanks{
The first author would like to express his appreciation to Professor Riccardo Salvati Manni for introducing this problem.
The second author would like to express his appreciation to 
Professor van Geemen for explaining the boundary of the Siegel threefold in question. 
The first author is supported by JSPS Grant-in-Aid for Scientific Research No. 24740017.
The second author is partially supported by JSPS Grant-in-Aid for Scientific Research No.23740027 and JSPS Postdoctoral 
Fellowships for Research Abroad No.378.}
\subjclass[2010]{}
\address{Takeo Okazaki, Department of Mathematics, Faculty of Science, Nara Woman University,
Kitauoyahigashi-machi, Nara 630-8506, Japan. }
\email{okazaki@cc.nara-wu.ac.jp}

\address{Takuya Yamauchi \\
Department of mathematics, Faculty of Education\\
Kagoshima University\\
Korimoto 1-20-6 Kagoshima 890-0065, Japan and\\ 
Department of mathematics \\
 University of Toronto \\
Toronto, Ontario M5S 2E4, CANADA} 
\email{yamauchi@edu.kagoshima-u.ac.jp or tyama@math.toronto.edu}

\begin{document}
\maketitle
\begin{abstract}
Let $Z$ be the quotient of the Siegel modular threefold $\mathcal{A}^{{\rm sa}}(2,4,8)$ which has been studied by 
van Geemen and Nygaard. They gave an implication that some 6-tuple $F_Z$ of theta constants which is in turn known to be a Klingen type Eisenstein series of weight 3 should be related to a holomorphic differential $(2,0)$-form on $Z$. The variety $Z$ is birationally equivalent to the tangent cone of Fermat quartic surface in the title. 

In this paper we first compute the L-function of two smooth resolutions of $Z$. One of these, denoted by $W$, is a kind of Igusa compactification 
such that the boundary $\partial W$ is a strictly normal crossing divisor. The main part of the L-function is described by some elliptic newform $g$ of 
weight 3. Then we construct an automorphic representation $\Pi$ of ${\rm GSp}_2(\A)$ related to $g$ and an explicit vector $E_Z$ sits inside $\Pi$ which 
creates a vector valued (non-cuspidal) Siegel modular form of weight $(3,1)$ so that $F_Z$ coincides with $E_Z$ in $H^{2,0}(\partial W)$ under 
the Poincar\'e residue map and various identifications of cohomologies. 


\end{abstract}
\section{Introduction}
For a  positive integer $N$, let $X_0(N)$ be the modular curves with respect to the congruence subgroup $\G^1_0(N)$ (cf. \cite{Shimura}) which 
has a canonical structure as a projective smooth curve over $\Q$. 
Then the L-function of the first $\ell$-adic etale cohomology $H^1_{{\rm et}}(X_0(N)_{\bQ},\Q_\ell)$ can be 
written in terms of automorphic L-functions associated to newforms in $S_2(\G^1_0(N))$ (see Section 7 of loc.cit.). 
Here $S_2(\G^1_0(N))$ is the space of elliptic cusp forms of weight two with respect to $\G^1_0(N)$. 

The modular curve $X_0(N)$ is a typical, basic example of Shimura variety of dimension one. 
So this phenomena makes us believe naively that Shimura varieties can be written in terms of automorphic 
L-functions. However, if we once move to the higher dimensional case, then the situation becomes much more difficult. 
There are many reasons here. 
Firstly we have to classify all automorphic representations in consideration to describe the L-functions of Shimura varieties. 
Secondly if the dimension of a Shimura variety in consideration is greater than one, then we have to study not only the 
cohomology of middle degree, but also it of another degrees except for zero and top. 

With this reasons it is important to possess many examples at hand to understand the cohomology of Shimura varieties. 
The case of Siegel modular threefolds which the authors are interested in seems to be one of most fascinating objects. 

Let $\mathcal{A}(2,4,8)$ be 
the moduli space of abelian surfaces 
with some level structure which has been studied by van Geemen, Nygaard, and van Straten (\cite{GS},\cite{vG-N}). 
It is the quotient space of the Siegel upper half plane $\mathbb{H}_2$ of degree 2  
by the arithmetic subgroup $\G(2,4,8)$ of the symplectic group $\Sp_2(\Z) \subset \GL_4(\Z)$. 
This congruence subgroup $\G(2,4,8)$ is contained in the principal congruence subgroup 
$\Gamma(4):=\{\gamma\in \Sp_2(\Z)\ |\ \gamma\equiv 1_4\ {\rm mod}\ 4 \}$ 
which is neat in the sense of Borel (Section 17 of \cite{borel}) and so is $\G(2,4,8)$. 
It follows from this that $\mathcal{A}(2,4,8)$ is a quasi-projective smooth threefold. 
By \cite{vG-N} we have the projective model ${\mathcal{A}}^{{\rm sa}}(2,4,8)$ so called the Satake compactification of $\mathcal{A}(2,4,8)$ which 
is defined over $\Q$ in $\P^{13}_\Q$ as follows:
$$\begin{array}{rl}
Y^2_0&=Q_0(X_0,X_1,X_2,X_3):=X^2_0+X^2_1+X^2_2+X^2_3\\
Y^2_1&=Q_1(X_0,X_1,X_2,X_3):=X^2_0-X^2_1+X^2_2-X^2_3\\
Y^2_2&=Q_2(X_0,X_1,X_2,X_3):=X^2_0+X^2_1-X^2_2-X^2_3\\
Y^2_3&=Q_3(X_0,X_1,X_2,X_3):=X^2_0-X^2_1-X^2_2+X^2_3\\
Y^2_4&=Q_4(X_0,X_1,X_2,X_3):=2(X_0X_1+X_2X_3)\\
Y^2_5&=Q_5(X_0,X_1,X_2,X_3):=2(X_0X_2+X_1X_3)\\
Y^2_6&=Q_6(X_0,X_1,X_2,X_3):=2(X_0X_3+X_1X_2)\\
Y^2_7&=Q_7(X_0,X_1,X_2,X_3):=2(X_0X_1-X_2X_3)\\
Y^2_8&=Q_8(X_0,X_1,X_2,X_3):=2(X_0X_2-X_1X_3)\\
Y^2_9&=Q_9(X_0,X_1,X_2,X_3):=2(X_0X_3-X_1X_2).
\end{array}
$$
They considered three kinds of quotient varieties of ${\mathcal{A}}^{{\rm sa}}(2,4,8)$ which are denoted by $X,Y,Z$ in loc.cit. 
and computed their Hodge numbers and the L-function of a smooth model of $X$ and $Y$. 
For $X$ (resp. Y) the middle (etale or de Rham) cohomology $H^3$ is related to a holomorphic Saito-Kurokawa lift 
(resp. endoscopic lifts). They also computed Hodge numbers of $Z$, but they 
left to study automorphic forms relate to holomorphic differential forms on $Z$ and any relation to the L-function of $Z$ 
(though they also computed a part of the L-function). Note that there are no holomorphic forms on $X$ and $Y$ 
other than holomorphic 3-forms (since $h^{2,0}=h^{1,0}=0$ in these cases). 
Contrary to $X$ and $Y$, there are no holomorphic 3-forms on $Z$, but $Z$ has a holomorphic 2-form as we will see below. 
This let the authors spur to study the various kinds of second (or fourth) cohomologies of $Z$ explicitly though 
the general results have been already built up by Weissauer \cite{wei1}, \cite{wei-coh} ,\cite{wei2} 
for any Siegel modular threefold with respect to a (principal) congruence subgroup. 

In this paper we first compute the L-function of $Z$ explicitly and secondly construct an explicit holomorphic 2-form on $Z$ 
related to the computation mentioned before.  
To explain the first main result, we need some notation. 
Let $\G_Z$ be a discrete subgroup of $\Sp_2(\Z)$ which is defined by $Z$ (see Remark \ref{Rem:defG_Z}). Then we will see that  
the Siegel threefold $\G_Z\bk\H_2$ is a smooth quasi-projective variety. 
Then the Satake compactification of $Z^\circ:=\G_Z \bk \H_2$, denoted by $Z=(Z^\circ)^{{\rm sa}}$, has the following defining equation in $\P^7_\Q$:
\begin{eqnarray}
Y^2_5=2(X_0X_2+X_1X_3)\\
Y^2_6=2(X_0X_3+X_1X_2)\\
Y^2_8=2(X_0X_2-X_1X_3)\\
Y^2_9=2(X_0X_3-X_1X_2)
\end{eqnarray}
(see lines 12-23 at p.55 of \cite{vG-N} and p.69-70 of loc.cit. for the 
defining equation).
The singular locus of this variety consists of the two lines defined by 
$$L_{i,i+1}:Y_5=Y_6=Y_8=Y_9=X_i=X_{i+1}=0,\ i=0,2.$$
Let $\wt{Z}$ be the resolution of $Z$ obtained by blowing up along $L_{i,i+1},\ i=0,2$. 
The Hodge numbers of $\widetilde{Z}$  is calculated in Proposition 2.14 of \cite{vG-N}:
\begin{equation}\label{hn}
h^{0,0}=h^{3,3}=h^{2,0}=h^{0,2}=1,\ h^{1,0}=h^{0,1}=h^{3,0}=h^{2,1}=h^{1,2}=h^{0,3}=0,\ h^{1,1}=23.
\end{equation}
We denote by $\widetilde{L}_{i,i+1}$ be the proper transform of $L_{i,i+1}$ in this blowing up. 
Unfortunately, $\widetilde{Z}$ is not a kind of Igusa's desingularization (cf. \cite{igusa}) of the Satake compactification of 
$Z^\circ=\G_Z\bk\H_2$ because $\widetilde{Z}\setminus Z^{\rm sm}=\widetilde{L}_{0,1}\coprod \widetilde{L}_{2,3}$ does not make a 
normal crossing divisor where $ Z^{\rm sm}$ is the smooth locus of $Z$.  
Further it is easy to see that $Z^\circ$ does not coincide with $Z^{\rm sm}=Z\setminus L_{0,1}\coprod L_{2,3}$. 

A correct modification $W$ of a resolution of $Z$ such that the boundary components consist of 
a normal crossing divisor and an explicit description of $W\setminus \partial W=Z^\circ$ will be given in Section \ref{description}. This will be needed to understand cohomology of 
$\G_Z\bk\H_2$. As we will see later, the variety $Z^\circ$ is a smooth, geometrically irreducible scheme over $\Z[\frac{1}{2}]$.  
Then we have the following result on the L-function of $Z^\circ$:
\begin{thm}\label{main1}(Theorem \ref{main-theorem2})
Keep the notation above. For a squarefree integer $a\neq 0,1$, let $\chi_a$ be the Dirichlet character defined by 
the quadratic residue symbol $\Big(\ds\frac{a}{\ast}\Big)$.
Let $g$ be the unique newform in $S_3(\G^1_0(16),\chi_{-1})$.
Then  
$$L(s,H^2_{{\rm et}}(Z^\circ_{\bQ},\Q_\ell))=\zeta(s-1)^{8}L(s-1,\chi_{-1})^7
L(s-1,\chi_{2})^2 L(s-1,\chi_{-2})^2 L(s,g).$$
Here $L(s,\chi_\ast)$ is the Dirichlet L-function for $\chi_\ast$ and $\zeta(s)$ is Riemann zeta function. 
In particular, the LHS is independent of any choice of $\ell$. 
\end{thm}
This main theorem follows from Theorem \ref{coh-deg2} with Proposition \ref{main-coh}. It follows from \cite{saito} 
the $\ell$-independence of the L-function. 
Note that $L(s,H^4_{{\rm et},c}(Z^\circ_{\bQ},\Q_\ell))=L(s-1,H^2_{{\rm et}}(Z^\circ_{\bQ},\Q_\ell))$ 
by Poincar\'e duality.

Since $H^3_!(Z^\circ,\Q)={\rm Im}(H^3_c(Z^\circ,\Q)\lra H^3(Z^\circ,\Q))=
{\rm Im}(H^3(\widetilde{Z},\Q)\lra H^3(Z^\circ,\Q))=0$ by (\ref{hn}) and Theorem 5.6 of \cite{o-s}, Theorem \ref{main1} tells us that 
most interesting part of the cohomology of $Z^\circ$ is $H^i,\ i=2,4$ contrary to the case $X$ and $Y$ 
(note that $H^{2,0}(X)=H^{2,0}(Y)=0$). 
Further it is quite natural to 
predict some relation between holomorphic differential forms on $Z^\circ$ (and also on $\partial W$) 
and our CM modular form $g$.  

To explain next main result we need more notation.  
Let $M^2_3(\G_Z)$ be the space of Siegel modular forms of weight 3 
with respect to $\G_Z$ and $M^2_3(\G_Z)^{KE}$ the subspace of $M^2_3(\G_Z)$ consisting of all Klingen type Eisenstein series. 
Let $M^2_{(3,1)}(\G_Z)$ be the space of vector valued Siegel modular forms of 
weight $\det^1{\rm St}_2\otimes {\rm Sym}^{2}{\rm St}_2$ with respect to $\G_Z$ (cf. \cite{arakawa}). 
Let $\widetilde{D}^{[1]}$ be the disjoint union of all irreducible components of $\partial W$.
As we will explain at Section \ref{diff}, there exist natural maps: 
\begin{equation}\label{eisen}
M^2_3(\G_Z)^{{\rm KE}}\hookrightarrow \ds\bigoplus_{i=0,2} H^{2,0}(\widetilde{L}_{i,i+1},\C)
\stackrel{{\rm rest}}{\hookleftarrow} H^{2,0}(W,\C)\simeq H^{2,0}(Z^\circ,\C) 
\stackrel{\sim}{\longleftarrow} M^2_{(3,1)}(\G_Z)
\end{equation}
where the first map is given by the projection to the component $\ds\bigoplus_{i=0,2} H^{2,0}(\widetilde{L}_{i,i+1},\C)$ of 
the composite of Eichler-Shimura embedding and the Poincar\'e residue map 
$H^3(Z^\circ,\C)\stackrel{{\rm res}[1]}{\lra}H^2(\widetilde{D}^{[1]},\C)=\C^{\oplus 40}\oplus 
\ds\bigoplus_{i=0,2} H^{2}(\widetilde{L}_{i,i+1},\C)$. Note that this map injects into 
$\ds\bigoplus_{i=0,2} H^{2,0}(\widetilde{L}_{i,i+1},\C)$.
The last map of (\ref{eisen}) is given by a natural identification $H^{2,0}(Z^\circ,\C)\simeq H^{2,0}(\widetilde{Z},\C)\simeq H^{2,0}(W,\C)$, 
since the holomorphic $2$-forms are uniquely extend to those on any 
smooth projective model of $Z^\circ$. 
Let $F_Z$ be the 6-tuple of theta constants defined  by (\ref{FZ}). By definition of $\G_Z$, our form $F_Z$ belongs to 
$M^2_3(\G_Z)$. 
Let $\widetilde{\G}_Z$ be the adelization of $\G_Z$ in $\GSp_2(\widehat{\Z})$ so that 
$\widetilde{\G}_Z\cap \Sp_2(\Q)=\G_Z$ which is introduced in Section \ref{construction}. 

From (\ref{eisen}) one can expect that there exists a vector valued Siegel modular form in $M^2_{(3,1)}(\G_Z)$ 
corresponding to $F_Z$ under (\ref{eisen})
(note that ${\rm dim}_\C M^2_{(3,1)}(\G_Z)={\rm dim}_\C H^{2,0}(Z,\C)=
{\rm dim}_\C H^{2,0}(\widetilde{Z},\C)=1$ by (\ref{hn})). 
Then we have the followings: 
\begin{thm}\label{main2}Keep the notation above. Let $g$ be the unique CM newform in $S^1_3(\G^1_0(16),\chi_{-1})$. 
Then 

\medskip
\noindent
$($i$)$ $F_Z$ is a Hecke eigen form with respect to Hecke operators for any $p\neq 2$ with the following Andrianov-Evdokimov's $L$-function (of degree $4$) outside $2$ (see (\ref{AE}) for the definition).

$$L^{(2)}(s,F_Z;{\rm AE})= \prod_{p \neq 2} L_p(s,g)L_p(s-1,g).$$

\medskip
\noindent
$($ii$)$  ${\rm dim}_\C M^2_3(\G_Z)^{KE}=1$ and $F_Z$ is a generator of $M^2_3(\G_Z)^{KE}$, 

\medskip
\noindent
$($iii$)$ there exists a non-cuspidal automorphic representation $\Pi$ of $\GSp_2(\A)$ and a smooth $\widetilde{\G}_Z$-fixed vector $E_Z$ such that 

\medskip
\noindent
$($a$)$ $(E_Z)_\infty$ has the highest weight vector in the minimal $K$-type $(3,1)$ as a representation of $U(2)$, 

\medskip
\noindent
$($b$)$ $E_Z$ gives a generator of $M^2_{(3,1)}(\G_Z)$ and it coincides with $F_Z$ in 
$\ds\bigoplus_{i=0,2} H^{2}(\widetilde{L}_{i,i+1},\C)$,
under the maps (\ref{eisen}).  

\medskip
\noindent
$($c$)$ $L^{(2)}(s-1,E_Z;{\rm AE})=L^{(2)}(s,F_Z;{\rm AE})=\ds\prod_{p\neq 2}L_p(s,g)L_p(s+1,g)$.

\end{thm}
Put $W_\ell=H^2_{{\rm et}}(Z^\circ)\oplus H^4_{{\rm et},c}(Z^\circ)$. Let $W^{{\rm t}}_\ell$ be the 
transcendental part of $W_\ell$.
\begin{Cor}\label{cor-of-main}
The following equality holds:
$$L^{(2)}(s,W^{{\rm t}}_\ell)=L^{(2)}(s-1,E_Z;{\rm AE})=L^{(2)}(s,F_Z;{\rm AE}).$$
\end{Cor}
We should mention what Theorem \ref{main2} insists on. 
Even if we have the equation (\ref{eisen}), we do not know a priori any direct 
relation between $M^2_3(\G_Z)$ and $M^2_{(3,1)}(\G_Z)$ because the map is just comparing 
the elements in question as an elliptic modular form. 
Further the shapes of Andrianov-Evdokimov (or spinor) L-functions of eigenforms of each space are different each other (see p.173, the last of Section 3.1 of \cite{arakawa}). 
Therefore what we have to do is firstly to compute the image $g$ of $F_Z$ under ${\rm res}[1]$ via Eichler-Shimura embedding which is 
nothing but the image under the Siegel $\Phi$-operator. Then next we try to find $E_Z\in M^2_{(3,1)}(\G_Z)$ 
related to $g$ so that it coincides with $F_Z$ under (\ref{eisen}). 
As a result, one has $L^{(2)}(s-1,E_Z;{\rm AE})=L^{(2)}(s,F_Z;{\rm AE})$. 
An interesting point is that  our form $E_Z$ of weight $(3,1)$ contribute 
to the mixed Hodge structure of $H^3(Z^\circ,\C)$ via $F_Z$ of weight $3$ which is an avatar of $E_Z$ in some sense. 
This might provokes us to consider the pure of weight 4 part of the mixed Hodge structure on the middle cohomology of a Siegel threefold 
in terms of the Klingen type Eisenstein series of weight $(3,1)$. 
Oda and Schwermer have already mentioned this kind of observation at the end of p.508, \cite{o-s}.  

However, the construction of $E_Z$ seems not to be easy as was done in \cite{arakawa} because 
the weight of $E_Z$ is small and therefore one will come across a problem on the convergence. 
Thus we need some arguments as in \cite{b&s} which make use of the method of 
the analytic continuation of real analytic Eisenstein series. However even if it is defined as in loc.cit., though 
one should extend the results to the vector valued case, it might be difficult to check the non-vanishing of it 
because the group structure of $\G_Z$ is slightly invisible due to 
the definition. Furthermore, it might be difficult to construct $E_Z$ starting from the classical setting (cf. p.63 of \cite{klingen})
because we have to customize $E_Z$ so that it has the central character $\chi_{-1}$ and belongs to 
$M_{(3,1)}(\G_Z)=M_{(3,1)}(\G(2),\chi_Z)$ simultaneously. 
 
To overcome these difficulties, we will apply the Soudry lift which is the theta lift from $\GO(2)$ to $\GSp(2)$, since the automorphic representation $\pi=\pi_g$ associated to the CM modular form $g$ comes from a grossencharacter on $\mathbb{A}^\times_K,\ K=\Q(\sqrt{-1})$. 
Then we extend the automorphic representation $\mu$ of $\GSO_K \simeq K^\t$ to that of $\GO_K = \GO(2)$.  
By using Soudry lift one can construct the desired irreducible representation $\Pi$ contributes to $M^2_{(3,1)}(\G_Z)\simeq H^{2,0}(Z^\circ,\C)=
H^{2,0}(Z^\circ,\C)\cap{\rm Eis}_{Q}$ where the last part is a part of $H^{2,0}(Z,\C)$ 
with respect to Klingen parabolic subgroup $Q$ (see \cite{Schwermer}). 
A theta lifting is one of powerful tools to create various automorphic forms on $\GSp_2(\A)$. 
However in our setting we have to take the level $\G_Z$ into account.  
The group structure of $\G_Z$ is somewhat invisible by the definition.
Therefore we need some observation of the theta kernel such that the image of the lifting is $\G_Z$-invariant. 
This will be devoted to Section \ref{construction}. We remark that the construction of Klingen type Eisenstein series 
of weight $(3,1)$ from CM elliptic modular forms has been already known for experts (see Section 6 of \cite{o-s} and Section 4 of \cite{wei-coh}). 

The paper is organized as follows. 
In Section 2, we study an algebraic description of $Z^\circ$ and determine the $L$-function of $Z^\circ$. 
Related to the results of Section 2, we discuss about the differential forms on $Z^\circ$ and a vector values Siegel modular form $E_Z$ related to $F_Z$ 
via maps between various cohomologies in Section 3. As we mentioned, we analyze the property of $F_Z$ and construct $E_Z$  in section 4.

\section{Notation}
For a prime number $p$ and a power $q$ of $p$, let $\F_q$ be the finite field with the cardinality $q$. 
For any $a\in \F_q$, we denote by $\Big(\ds\frac{a}{q}\Big)$ the quadratic residue symbol of $a$. 
For a finite set $X$, we denote by $|X|$ its cardinality. 
For a commutative algebra $R$, let $R^\t$ denote the group of units of $R$.
Let $R^1$ denote the group of elements of norm $1$ whenever a norm is once given on $R$. 
Throughout this paper, $\ell$ means any prime number. 

\section{two smooth models of $Z$ and its cohomologies}
In this section we will study two smooth resolutions $\widetilde{Z}$ and $W$ of $Z$ as we mentioned in Section 1 and its arithmetic. 
We will also study the mixed Hodge structure on the cohomologies of $Z$ and its relation to (non-cuspidal) 
Siegel modular forms with respect to the discrete group $\G_Z$ defined by $Z$. 
Throughout this section we always assume a prime $p$ (and hence its power $q$) to be odd. 
We refer to Section 4 for the notation of modular forms in various settings which we need in this paper. 

\subsection{L-function of $\widetilde{Z}$} 
In this subsection we shall compute the L-function of $\widetilde{Z}$. 
To do this we need to take a few steps to reduce the defining equation of $Z$ to more 
convenient one. Most materials here have already obtained in \cite{vG-N}, but we need some modifications. 
We recall the defining equation of $Z$ again and replace the original coordinates $Y_5,Y_6,Y_8,Y_9$ by $Y_3,Y_0,Y_2,Y_1$ 
for simplicity:
\begin{eqnarray}
Y^2_0=2(X_0X_3+X_1X_2)\\
Y^2_1=2(X_0X_3-X_1X_2)\\
Y^2_2=2(X_0X_2-X_1X_3)\\
Y^2_3=2(X_0X_2+X_1X_3).
\end{eqnarray}

Henceforth we consider the all geometric objects or morphisms betweem them as $\Z[\frac{1}{2}]$-schemes or 
$\Z[\frac{1}{2}]$-morphisms and we will freely use the basic facts on etale cohomology (cf. \cite{milne}). 

Let $\P^n$ be the projective space of dimension $n$ with the fixed coordinates $[Z_0:\cdots:Z_n]$ and 
$F$ the quartic Fermat surface defined by $Z^4_0-Z^4_1+Z^4_2-Z^4_3=0$ in $\P^3$. 
Let ${\rm Cone}_\infty(F)$ be the tangent cone of $F$ at infinity in $\P^4$ which is defined by 
$${\rm Cone}_\infty(F):=\{[t:Z_0:Z_1:Z_2:Z_3]\in\P^4\ |\ Z^4_0-Z^4_1+Z^4_2-Z^4_3=0\}.$$
As in Proposition 2.13 of \cite{vG-N}, the birational map $\phi:{\rm Cone}_\infty(F)\lra Z$ is given by 
sending $[t:Z_0:Z_1:Z_2:Z_3]$ to 
$$[2Z_0:2Z_1:2Z_2:2Z_3:\frac{Z^2_0+Z^2_1}{t}:\frac{-Z^2_2+Z^2_3}{t}:\frac{t(Z^2_2+Z^2_3)}{Z^2_0+Z^2_1}:t].$$
Note that the constant $c$ in Proposition 2.13 of \cite{vG-N} should be $2$. 
It is easy to see that the converse is given by 
sending $[Y_0:Y_1:Y_2:Y_3:X_0:X_1:X_2:X_3]$ to 
$[X_3:\frac{Y_0}{2}:\frac{Y_1}{2}:\frac{Y_2}{2}:\frac{Y_3}{2}]$. 
For abbreviation we write the coordinates $[X_0:X_1:X_2:X_3]$ as $\underline{X}$ and it is the same as well 
for $\underline{Y}$ and $\underline{Z}$ 

The following proposition gives a modification of Proposition 2.13 of \cite{vG-N}.  
\begin{prop}\label{iso}Let $U_1=\{[t:\underline{Z}]\in {\rm Cone}_\infty(F)\ |\ t\neq 0\ {\rm and}\ Z^2_0+Z^2_1\neq 0 \}$ and 
$U_2=\{[\underline{Y}:\underline{X}]\in Z\ |\ Y^2_0+Y^2_1\neq 0 \ {\rm and}\ X_3\neq 0 \}$. 
Then the birational map $\phi$ gives an isomorphism $U_1\stackrel{\sim}{\lra} U_2$ as an open, 
geometrically irreducible $\Z[\frac{1}{2}]$-scheme.
\end{prop}
\begin{proof}
It suffices to check the equivalence that $Z^2_0+Z^2_1\neq 0\Longleftrightarrow Y^2_0+Y^2_1\neq 0$.
\end{proof} 
For $i=1,2$, we denote by $U^c_1$ (resp. $U^c_2$) the reduced closed subscheme ${\rm Cone}_\infty(F)\setminus U_1$ (resp. 
$Z\setminus U_2$). 
\begin{prop}\label{point1}The following equalities hold:   

\medskip
\noindent
(1) $|U^c_1(\F_q)|=|F(\F_q)|+4q^2+4q^2\Big(\ds\frac{-1}{q}\Big)-4q-6q\Big(\ds\frac{-1}{q}\Big)+1+2\Big(\ds\frac{-1}{q}\Big)$,  

\medskip
\noindent
(2) $|U^c_2(\F_q)|=4q^2-2q+2+(4q^2-6q+2)\Big(\ds\frac{-1}{q}\Big)$, 

\medskip
\noindent
(3) $|{\rm Cone}_\infty(F)(\F_q)|=q|F(\F_q)|+1$.

\end{prop}
\begin{proof}
We first prove (2). First we observe 
$$|U^c_2(\F_q)|=|\{Y^2_0+Y^2_1=0\}|+|\{X_3=0\}|-
\Big|\left\{
\begin{array}{cc}
Y^2_0+Y^2_1=0\\
X_3=0
\end{array}
\right\}\Big|.$$
Since the last two terms are both equal to each other, one has 
$$|U^c_2(\F_q)|=|\{Y^2_0+Y^2_1=0\}|=|\{X_0X_3=0\}|=
|\{X_0=0\}|+|\{X_0=1,\ X_3=0\}|.$$ 
If $X_0=0$, then the defining equation is as follows:
\[
\begin{array}{rl}
Y^2_0&=2X_1X_2\\
Y^2_1&=-2X_1X_2\\
Y^2_2&=-2X_1X_3\\
Y^2_3&=2X_1X_3.
\end{array}
\]
Assume $X_1=0$, then the number in consideration is nothing but 
it of all $[X_2:X_3]\in \P^1(\F_q)$, hence it is $q+1$. 
If $X_1\neq 0$, then the number in consideration amounts to 
 \[
\begin{array}{rl}
&\ds\sum_{X_2,X_3\in\F_q}\Big(1+\Big(\frac{2X_2}{q}\Big) \Big)\Big(1+\Big(\frac{-2X_2}{q}\Big) \Big)
\Big(1+\Big(\frac{-2X_3}{q}\Big) \Big)\Big(1+\Big(\frac{2X_3}{q}\Big) \Big)\\
\\
=& \Big\{\ds\sum_{X_2\in\F_q}\Big(1+\Big(\frac{2X_2}{q}\Big) \Big)\Big(1+\Big(\frac{-2X_2}{q}\Big) \Big)  \Big\}^2
\\
=& \Big\{\ds\sum_{X_2\in\F_q}\Big(1+\Big(\frac{-2X_2}{q}\Big)+\Big(\frac{2X_2}{q}\Big)+\Big(\frac{-4X^2_2}{q}\Big) \Big)  \Big\}^2
\\
=& \Big\{q+(q-1)\Big(\frac{-1}{q}\Big)  \Big\}^2=2q^2-2q+1+(2q^2-2q)\Big(\frac{-1}{q}\Big).
\end{array}
\]
Summing up, one has $|\{X_0=0\}|=2q^2-q+2+(2q^2-2q)\Big(\ds\frac{-1}{q}\Big).$

Assume that $X_0=1$ and $X_3=0$. Then one has 
 \[
\begin{array}{rl}
|\{X_0=1,\ X_3=0\}|=&\ds\sum_{X_1,X_2\in\F_q}\Big(1+\Big(\frac{2X_1X_2}{q}\Big) \Big)\Big(1+\Big(\frac{-2X_1X_2}{q}\Big) \Big)
\Big(1+\Big(\frac{2X_2}{q}\Big) \Big)^2\\
\\
=&q+\ds\sum_{X_1,X_2\in\F_q\atop 
X_2\neq 0}\Big(1+\Big(\frac{2X_1}{q}\Big) \Big)\Big(1+\Big(\frac{-2X_1}{q}\Big) \Big)
\Big(1+\Big(\frac{2X_2}{q}\Big) \Big)^2
\\
=& q+\ds\sum_{X_1\in\F_q}\Big(1+\Big(\frac{2X_1}{q}\Big) \Big)\Big(1+\Big(\frac{-2X_1}{q}\Big) \Big)
\ds\sum_{X_2\in\F_q\atop 
X_2\neq 0}\Big(1+\Big(\frac{2X_2}{q}\Big) \Big)^2\\
\\
=&q+\Big\{q+(q-1)\Big(\ds\frac{-1}{q}\Big)\Big\}\times 2(q-1)=2q^2-q+(2q^2-4q+2)\Big(\frac{-1}{q}\Big).
\end{array}
\]
This claims (2). 

For (1), one has 
$$|U^c_1(\F_q)|=|\{t=0\}|+|\{Z^2_0+Z^2_1=0\}|-|\{t=0,\ Z^2_0+Z^2_1=0\}|.$$ 
The first term is $|F(\F_q)|$. 
For the second term, if $Z_0=0$, then $Z_1=0$. 
So the number of $\{[t:0:0:Z_2,Z_3]\in {\rm Cone}(F)(\F_q)\}$ is 
$p\Big(3+\Big(\ds\frac{-1}{q}\Big)\Big)+1$. 
Here we use the formula $|\{x\in \F_q\ |\ x^4=1\}|=3+\Big(\ds\frac{-1}{q}\Big)$. 
If $Z_0=1$, then the number of $\{[t:1:Z_1:Z_2:Z_3]\in {\rm Cone}(F)(\F_q)\}$ is 
$q\Big(1+\Big(\ds\frac{-1}{q}\Big)\Big)\Big\{1+(q-1)\Big(3+\Big(\ds\frac{-1}{q}\Big)\Big)\Big\}$. This gives us the first assertion (1). 

The last claim (3) follows from the definition of the tangent cone. 
This completes the proof. 
\end{proof}
\begin{Cor}\label{point2} The equality $|Z(\F_q)|=(q-1)|F(\F_q)|+2q+2$ holds.
\end{Cor}
\begin{proof}
Since $|Z(\F_q)|=|{\rm Cone}_\infty(F)(\F_q)|+|U^c_2(\F_q)|-|U^c_1(\F_q)|$, the claim follows from 
Proposition \ref{iso}, \ref{point1}. 
\end{proof}

We next study a resolution of singularities of $Z$. As mentioned in Section 1, 
this variety has singularities along with the two lines defined by 
$$L_{i,i+1}:Y_0=Y_1=Y_2=Y_3=X_i=X_{i+1}=0,\ i=0,2$$
(note that we have changed the numbering of the subscripts of $Y_\ast$).
Let $\pi:\widetilde{Z}\lra Z$ be the resolution of $Z$ obtained by blowing up along $L_{i,i+1},\ i=0,2$. 
Clearly $\widetilde{Z}$ is defined over $\Q$ and even it can be regarded as a smooth $\Z[\frac{1}{2}]$-scheme. 
\begin{prop}\label{exdiv}The notation is same as above. Let $\widetilde{L}_{i,i+1}$ be the proper  
transform of $L_{i,i+1}$ for $i=1,2$. 
Then $\widetilde{L}_{i,i+1}$ is isomorphic to $F$ as a $\Z[\frac{1}{2}]$-scheme and they never intersect each other. 
\end{prop}
\begin{proof}Note that $L_{0,1}\simeq \P^1$ with coordinates $[X_2:X_3]$. 
By the proof of Proposition 2.12-(2) in \cite{vG-N}, $\widetilde{L}_{0,1}$ is given by 
\[
\begin{array}{c}
(Z^2_0+Z^2_1)X_2=(Z^2_2+Z^2_3)X_3\\
(-Z^2_2+Z^2_3)X_2=(Z^2_0-Z^2_1)X_3
\end{array}
\]
in $\P^1\times \P^3$ with coordinates $([X_2:X_3],[\underline{Z}])$. 
One can see easily that the natural projection to $\P^3$ gives an isomorphism 
$\widetilde{L}_{0,1}\stackrel{\sim}{\lra} F$. 
It is the same for $\widetilde{L}_{2,3}$. 

The last claim follows from that $L_{0,1}\cap L_{2,3}=\emptyset$. 
\end{proof}

\begin{Cor}\label{resolution}The notation is same as above. Then 
$$|\widetilde{Z}(\F_q)|=(q+1)|F(\F_q)|.$$
\end{Cor}
\begin{proof}
Since $|\widetilde{Z}(\F_q)|=|Z(\F_q)|+\ds\sum_{i=0,2}|\widetilde{L}_{i,i+1}(\F_q)|-\sum_{i=0,2}|L_{i,i+1}(\F_q)|$, the claim follows from 
Corollary \ref{point2} and Proposition \ref{exdiv}. 
\end{proof}
\begin{Rem} The Fermat quartic $F$ gives a model over $\Q$ of the Shioda's elliptic modular surface of level 4 (cf. p.71 of \cite{vG-N}). 
From this one has $H^{2,0}(F,\C)\simeq S^1_3(\G^1(4))\simeq S^1_3(\G^1_0(16),\chi_{-1})$. 
\end{Rem}
We now compute the number of $\F_q$-rational points of $F$. 
This seems to be a well-known result for experts, but we give a proof here 
because we need to take keeping the base field (that is $\Q$) into account to our purpose. 
\begin{prop}\label{fermat}
Let $q$ be a prime power which is coprime to 2 and let $g=\ds\sum_{n\ge 1}a_n(g)q^n$ be the CM modular form $g$ in Theorem \ref{main1}. 
Then 
$$|F(\F_q)|=9q+7\Big(\ds\frac{-1}{q}\Big)q+2\Big(\ds\frac{2}{q}\Big)q+
2\Big(\ds\frac{-2}{q}\Big)q+a_q(g).$$
\end{prop}
\begin{proof}We may assume $q=p$. 
Let $C$ be the Fermat quartic defined by $x^4_0+x^4_2=x^4_1$ in $\P^2$ with coordinates $[x_0:x_1:x_2]$. 
Consider a generically finite, rational map 
$$\varphi:C\times C\lra F; ([x_0:x_1:x_2],[y_0:y_1:y_2])\mapsto 
[x_0y_2:x_1y_2:x_2y_1:x_2y_0].$$
By Corollary 2.6, p.126 of \cite{otsubo}, one has the decomposition 
\begin{equation}\label{decomp}
H^2_{{\rm et}}(F_{\bQ},\Q_\ell)=\Q_\ell(-1)^{\oplus 9}\oplus
\Q_\ell(-1)(\chi_{-1})^{\oplus 7}
\oplus
\Q_\ell(-1)(\chi_{2})^{\oplus 2}
\oplus
\Q_\ell(-1)(\chi_{-2})^{\oplus 2}\oplus V_{\ell}
\end{equation}
where $V_\ell$ is a 2-dimensional $\ell$-adic geometric representation of $G_{\Q}:={\rm Gal}(\bQ/\Q)$ 
with Hodge-Tate weights $\{2,0\}$ and $\chi_a:G_\Q\lra {\bQ}^\times_\ell$ is the 
quadratic character associated to the quadratic extension $\Q(\sqrt{a})/\Q$ for $a=-1,2,-2$. 
Suppose that $V_\ell$ is reducible then the semi-simple part becomes $V^{{\rm ss}}_\ell=\Q_\ell(2)\varepsilon_1\oplus \Q_\ell \varepsilon_2$ 
for some finite characters $\varepsilon_i:G_\Q\lra {\bQ}^\times_\ell,\ i=1,2$ which are unramified outside 2. 
For any isomorphism $\iota:\bQ_\ell\lra \C$ as a field and any prime $p>2$, we have 
$$|p^2\iota(\varepsilon_1(p))|-|\iota(\varepsilon_2(p))|=p^2-1\le {\rm tr}({\rm Frob}_p|V_\ell)\le 2p$$
by the Weil bound. But this contradicts to $p>2$. 
Hence $V_\ell$ is irreducible.

By Table 8, p.454 of \cite{shimura-mahoro}, $C\simeq X_0(64)$ over $\Q$ and 
$J_0(64):={\rm Jac}(X_0(64))=\stackrel{\Q}{\sim}E^2_{32}\times E_{64}$ where 
$E_{32}:zy^2=x^3-xz^2$ and $E_{64}:zy^2=x^3+xz^2$ are elliptic curves with conductors 32 and 64 respectively. 
The curve $E_{64}$ is the quadratic twist of $E_{32}$ by the quadratic field $\Q(\sqrt{-1})/\Q$.
Then by looking pull-back of $\varphi$, the 2-dimensional irreducible quotient of $H^2_{{\rm et}}(C\times C_{\bQ},\Q_\ell)$ with 
Hodge-Tate weight $\{2,0\}$ is a sub-quotient of $\wedge^2H^1_{{\rm et}}({E_{32}}_{\bQ},\Q_\ell)\simeq \wedge^2 H^1_{{\rm et}}({E_{64}}_{\bQ},\Q_\ell)$. Note that the isomorphism comes from the fact $H^1_{{\rm et}}({E_{32}}_{\bQ},\Q_\ell)
\simeq H^1_{{\rm et}}({E_{64}}_{\bQ},\Q_\ell)(\chi_{-1})$. 
From this, one can see that $V_\ell$ is the $\ell$-adic realization of a CM motive of rank two and it has to correspond to 
a CM elliptic newform $h$ of weight 3. 
By Lemma 3 of \cite{shimura-cm}, the level of $h$ is at most 32. 
Hence $h$ is in $S^1_3(\G^1_0(16),\chi_{-1})$ or $S^1_3(\G^1_0(32),\chi_{-1})$ by Stein's table \cite{stein}. 
Compairing the Fourier coefficients at $p=3$, $h$ should be our $g$ since ${\rm tr}({\rm Frob}_3|V_\ell)=-6$. Here we use the 
formula  
$${\rm tr}({\rm Frob}_p|V_\ell)=|F(\F_p)|-(p^2+1)-\Big\{9p+7\Big(\ds\frac{-1}{p}\Big)p+2\Big(\ds\frac{2}{p}\Big)p+
2\Big(\ds\frac{-2}{p}\Big)p\Big\}$$ and 
$|F(\F_3)|=16$. Hence one has $V_\ell\simeq V_{g,\ell}$.  
\end{proof}
Finally we study the etale cohomology of $\widetilde{Z}$. 
\begin{prop}\label{main-coh} For any $i$, $H^i_{{\rm et}}(\widetilde{Z}_{\bQ},\Q_\ell)$ is a semi-simple $\Q_\ell[G_\Q]$-module and 
it is given as follows: 
\[
H^i_{{\rm et}}(\widetilde{Z}_{\bQ},\Q_\ell)=
\left\{\begin{array}{cc}
\Q_\ell & (i=0) \\
\Q_\ell(-1)\oplus H^2_{{\rm et}}(F_{\bQ},\Q_\ell) & (i=2)  \\
\Q_\ell(-2)\oplus H^4_{{\rm et}}(F_{\bQ},\Q_\ell) & (i=4)  \\
\Q_\ell(-3) & (i=6) \\
0 & (i=1,3,5)  
\end{array}\right.
\]
where 
$$H^2_{{\rm et}}(F_{\bQ},\Q_\ell)=H^4_{{\rm et}}(F_{\bQ},\Q_\ell)(1)=
\Q_\ell(-1)^{\oplus 9}\oplus
\Q_\ell(-1)(\chi_{-1})^{\oplus 7}
\oplus
\Q_\ell(-1)(\chi_{2})^{\oplus 2}
\oplus
\Q_\ell(-1)(\chi_{-2})^{\oplus 2}\oplus V_{g,\ell}
$$
as a $\Q_\ell[G_\Q]$-module. 
\end{prop}
\begin{proof}
To prove this, it suffices to prove it only for $i=2$. Let ${\rm NS}(\widetilde{Z})_{\bQ}$ be the Neron-Severi group generated by 
cycles defined over $\bQ$. 
Then ${\rm NS}(\widetilde{Z})_{\bQ}\otimes_\Z\Q$ exhausts $H^{1,1}(\widetilde{Z},\Q)$ by Proposition 2.14 of \cite{vG-N}. 
Let ${\rm cl}_\ell:{\rm NS}(\widetilde{Z})_{\bQ}\otimes_\Z \Q_\ell\lra H^2_{{\rm et}}(\widetilde{Z}_{\bQ},\Q_\ell)$ be the 
$\ell$-adic cycle map. 
Then by comparison theorem, Corollary \ref{resolution}, and Proposition \ref{fermat}, 
one has $H^2_{{\rm et}}(\widetilde{Z}_{\bQ},\Q_\ell)={\rm Im}({\rm cl}_\ell)\oplus V_{g,\ell}$ as a $\Q_\ell[G_\Q]$-module, hence giving the claim.
\end{proof}

\subsection{An algebraic description of $Z^\circ$}\label{description}
In previous section we have studied arithmetic of $\widetilde{Z}$. To make a connection to Siegel modular forms on $\G_Z$ more precisely, we have 
to investigate an algebraic description of $Z^\circ=\G_Z\backslash \H_2$ and the difference between $Z^\circ$ and $Z=(Z^\circ)^{{\rm sa}}$. We are starting from the following key Lemma. To explain this we need more notation.
Let $\mathcal{A}(2,4)$ be the Siegel modular threefold associated to $\G(2,4)$ (see Section 1 of \cite{vG-N}). Then its Satake compactification is isomorphic to 
$\P^3$ (Proposition 1.7 of \cite{vG-N}). Then the natural projection $\pi:Z\lra \mathcal{A}(2,4)^{{\rm sa}}\simeq \P^3$ induced by 
the inclusion $\G_Z\subset \G(2,4)$ is given by 
$[\underline{X}:\underline{Y}]\mapsto [\underline{X}]$. 

We introduce the following $30=4\times 6+6$ lines in $\P^3$:
\begin{equation}
\begin{array}{l}
L^{(\pm,\pm)}_1=\{[\pm \sqrt{-1}X_2:\pm \sqrt{-1}X_3:X_2:X_3]\in \P^3\}\stackrel{\Q(\sqrt{-1})}{\simeq} \P^1_{\Q(\sqrt{-1})}\\
L^{(\pm,\pm)}_2=\{[\pm \sqrt{-1}X_1:X_1:\pm \sqrt{-1}X_3:X_3]\in \P^3\}\stackrel{\Q(\sqrt{-1})}{\simeq} \P^1_{\Q(\sqrt{-1})}\\
L^{(\pm,\pm)}_3=\{[\pm \sqrt{-1}X_3:\pm \sqrt{-1}X_2:X_2:X_3]\in \P^3\}\stackrel{\Q(\sqrt{-1})}{\simeq} \P^1_{\Q(\sqrt{-1})}\\
L^{(\pm,\pm)}_4=\{[\pm X_3:\pm X_2:X_2:X_3]\in \P^3\}\stackrel{\Q}{\simeq} \P^1_{\Q}\\
L^{(\pm,\pm)}_5=\{[\pm X_1:X_1:\pm X_3:X_3]\in \P^3\}\stackrel{\Q}{\simeq} \P^1_{\Q}\\
L^{(\pm,\pm)}_6=\{[\pm X_2:\pm X_3:X_2:X_3]\in \P^3\}\stackrel{\Q}{\simeq} \P^1_{\Q}\\
\\
l_{i,j}:X_i=X_j=0\ \mbox{for $0\le i<j\le 3$ }.
\end{array}
\end{equation}
\begin{lem}\label{line1}The boundary $\partial \mathcal{A}(2,4)^{{\rm sa}}=\mathcal{A}(2,4)^{{\rm sa}}\setminus \mathcal{A}(2,4)$ consists of  
the above 30 lines.
\end{lem}
\begin{proof}
This is already done at p.53-54, the proof of Proposition 1.7 of \cite{vG-N}. 
By the proof there, one can see that the boundary consists of the lines which obtained by the intersections of all 
pair of ten quadratic equations $Q_i(X_0,X_1,X_2,X_3)=0,\ (0\le i\le 9)$ in $\P^3$ (See Section 1 for $Q_i$).  
\end{proof}
We now determine the boundary $Z\setminus Z^\circ$. 
Let $D^{\pm,\pm}_i=\pi^{-1}(L^{\pm,\pm}_i)^{{\rm red}}$ for $1\le i \le 6$ where the superscript ``red" means the reduced scheme structure and  
$L_{i,j}=\pi^{-1}(l_{i,j})^{{\rm red}}$. For $i=1,3,4,6$, $D^{\pm,\pm}_{i}$ is isomorphic over $\Q(\zeta_8)$ to the modular curve $X(8)=\G^1(8)\backslash \mathbb{H}_1$ of genus 5. 
 For $k=2,5$ and $(i,j)=(0,2),(0,3),(1,2),(1,3)$, each of $D^{\pm,\pm}_{k}$ and $L_{i,j}$ consists of two irreducible 
divisors defined over $\Q(\zeta_8)$ 
which are isomorphic to $\P^1$ and intersect two points. The lines $L_{0,1}$ and $L_{2,3}$ are nothing but those introduced in Section 1. 
Put $D=\ds\bigcup_{k=1}^6 D^{\pm,\pm}_k\cup \bigcup_{1\le i< j\le 4}L_{i,j}$. 
\begin{prop}\label{line2}
The boundary $Z\setminus Z^\circ$ is given by $D$ and it consists of 42 irreducible components. 
Further $D$ can be regarded as a scheme over $\Z[\frac{1}{2}]$.  
\end{prop}
\begin{proof}Recall that $Z$ and $\mathcal{A}(2,4)^{{\rm sa}}$ are both the Satake compactifications and these are 
 defined by the theta embeddings which are compatible with natural projection 
$\G_Z\backslash \mathbb{H}_2\lra \mathcal{A}(2,4)$ induced from the inclusion $\G\subset \G(2,4)$. Therefore by definition of Satake compactification, 
the boundary should be given by $\pi^{-1}(\partial \mathcal{A}(2,4)^{{\rm sa}})^{{\rm red}}$. 
Then by writing down all divisors explicitly, one has the claims.   
\end{proof}
\begin{Cor}The variety $Z^\circ=Z\setminus D$ is a smooth algebraic variety and it is an algebraic description of 
$Z^\circ(\C)=\G_Z\backslash\H_2$. Further $Z^\circ$ can be regarded as a smooth, geometrically irreducible scheme over $\Z[\frac{1}{2}]$.  
\end{Cor}
\begin{proof}
These are easy to follow from Proposition \ref{line2} and the defining equation of $Z$ in Section 1.
\end{proof}
\begin{Rem}\label{Rem:defG_Z}
The variety $Z$ is defined by the Zariski closure of the image of the map from $\alpha_Z:\mathbb{H}_2\lra \mathbb{P}^7$ by using the 
following theta functions:
$$\theta_{(1,0,0,0)}(\tau),\ \theta_{(1,1,0,0)}(\tau),\ \theta_{(1,0,0,1)}(\tau),\ \theta_{(1,1,1,1)}(\tau),$$
$$\theta_{(0,0,0,0)}(2\tau),\ \theta_{(0,1,0,0)}(2\tau),\ \theta_{(1,0,0,0)}(2\tau),\ 
\theta_{(1,1,0,0)}(2\tau),$$ 
(see section \ref{construction} for the definition of these theta constants $\th$). 
Then $\G_Z$ is defined by the subset of $\G(2,4)$ consisting of an elements $g$ such that 
$g$ fixes $\alpha_Z$. This means that for such $g$, there exists a non-zero constant $c_g$ such that 
$$\theta_\ast(\tau)|[g]=c_g \theta_\ast(\tau),\ \ast\in \{(1,0,0,0),(1,1,0,0),(1,0,0,1),(1,1,1,1)\}$$ and 
$$\theta_\ast(2\tau)|[g]=c_g \theta_\ast(2\tau),\ \ast\in \{(0,0,0,0),(0,1,0,0),(1,0,0,0),(1,1,0,0)\}$$

Using the transformation formula (cf. Lemma \ref{lem:igusa}), one can find that $\G_Z$ is generated by $\G(4,8)$ and 
$$e_1e_4,\ e_1e_6,\ e_1e^2_9,\ e^2_8e_3,\ e_2e^2_{10}$$
in $\G(2,4)$ where $e_i$'s are defined in Section \ref{construction}. 
Therefore, by Proposition \ref{prop:table}, a $2r$-tuple product $\prod_{j=1}^{2r} \th_{\m_j}$ of even or odd theta constants belongs to $M_{r}^2(\G_Z)$, if and only if 
\begin{eqnarray*}
\sum_j b_j c_j + \sum_j c_j d_j \equiv \sum_j b_j c_j + \sum_j c_j \equiv \sum_j b_j + \sum_j a_j d_j \equiv \sum_j a_jd_j + \sum_j d_j \\
\equiv \sum_j b_jc_j + \sum_j a_jc_j -1 \equiv 0 \pmod{2} 
\end{eqnarray*}
where $\m_j = (a_j,b_j,c_j,d_j)$.
\end{Rem}
Finally we discuss another smooth compactification of $Z^\circ$ so that the boundary consists of normal crossing divisors. 
It is straightforward that each irreducible component of $D$ necessarily intersects transversally another component of $D$. 
Further only three components pass at an intersection point. More precisely, for such a point $p_0\in D$, 
the etale neighborhood of $D$ at $p_0$ is isomorphic to $\{x=y=0\}\ \cup \{y=z=0\}\ \cup \{z=x=0\}$ in 
$\A^3={\rm Spec}(\C[x,y,z])$. 
Furthermore, three components of $D$ which are passing an intersection point never be lying on the same plane.  
From this 
if we denote by $W$ the resolution of $Z$ along the components of $D$,  then one can see that 
the strict transform $\widetilde{D}$ of $D$ forms a strictly normal crossing divisors of $W$. 
This should be called a kind of Igusa compactification. We finish this subsection with studying the cohomologies of $W$. 
\begin{prop}\label{W}$H^2(W,\C)=H^2(\widetilde{Z},\C)\oplus \C^{\oplus 40}$ as a Hodge module 
$($ hence $\C$ is of Hodge type $(1,1))$ and in particular, 
$H^{2,0}(W,\C)=H^{2,0}(\widetilde{Z},\C)\simeq H^{2,0}(F,\C)\simeq S^1_3(\G^1(4))$.
\end{prop}
\begin{proof}This follows from the result at p.443 of \cite{manin}. 
\end{proof}

\subsection{L-function of $Z^\circ$}\label{L2}
In this subsection we study the $\ell$-adic cohomology of the second degree of $Z^\circ$ and in particular, 
determine the L-function of $H^2_{{\rm et}}(Z^\circ_{\bQ},\Q_\ell)(-1)\simeq H^4_{{\rm et,c}}(Z^\circ_{\bQ},\Q_\ell)$. 
The following fact  will be used soon later. We denote by $H^\ast_c$ the Betti (singular) cohomology with compactly support. 
\begin{prop}\label{vanishing}For any smooth variety $Z'$ and an open immersion $Z^\circ\hookrightarrow Z'$, one has 
$H^5_c(Z^\circ,\Q)=H^5_c(Z',\Q)=0$. 
\end{prop}
\begin{proof}
First of all we prove that $H^5_c(Z^\circ,\Q)=0$. 
Since ${Z^\circ}(\C)\simeq \G_Z\backslash\H_2$, then 
it follows from Poincar\'e duality that $H^5_c({Z^\circ}(\C),\Q)\simeq H_1(\G_Z\backslash \H_2,\Q)=\G_Z/[\G_Z,\G_Z]\otimes_\Z \Q=0$. 

Put $D:=Z'\setminus Z^\circ$. Then $H^5_c(D,\Q)=0$ since the dimension of $D$ as an algebraic variety is at most 2.  
Then the claim follows from the long exact sequence 
$$\cdots\lra H^5_{c}(Z^\circ,\Q)\lra 
 H^5_{c}(Z',\Q)\lra H^5_{c}(D,\Q)$$
with $H^5_{c}(Z^\circ,\Q)=H^5_{c}(D,\Q)=0$. 
\end{proof}
\begin{Cor}\label{vanishing-et}For any smooth variety $Z'$ over $\Q$ and an open immersion $Z^\circ\hookrightarrow Z'$ over $\Q$, one has 
$H^5_{{\rm et,c}}(Z^\circ_{\bQ},\Q_\ell)=H^5_{{\rm et,c}}({Z'}_{\bQ},\Q_\ell)=0$. 
\end{Cor} 
\begin{proof}
Since $Z^\circ$ and $Z'$ are both smooth, by the comparison theorem between etale and singular cohomologies,  one has the claim with 
Proposition \ref{vanishing}.
\end{proof}
We now compute the cohomology of $H^4_{{\rm et,c}}(Z^\circ,\Q_\ell)$. 
\begin{thm}\label{coh-deg2}The following description holds: 
$$H^2_{{\rm et}}(Z^\circ_{\bQ},\Q_\ell)(-1)\simeq H^4_{{\rm et,c}}(Z^\circ_{\bQ},\Q_\ell)=
 H^4_{{\rm et}}(\widetilde{Z}_{\bQ},\Q_\ell)/\Q_\ell(-2)^{\oplus 2}.$$
\end{thm}
\begin{proof}
We first compute $H^4_{{\rm et,c}}({Z^{{\rm sm}}}_{\bQ},\Q_\ell)$. 
Recall that $\widetilde{Z}\setminus Z^{\rm sm}=\widetilde{L}_{0,1}\coprod \widetilde{L}_{2,3}$ and 
$\widetilde{L}_{i,i+1},\ i=0,1$ is isomorphic to the Fermat quartic $F$ over $\Q$. Consider the following exact sequence 
$$H^3_{{\rm et,c}}((\widetilde{L}_{0,1}\coprod \widetilde{L}_{2,3})_{\bQ},\Q_\ell)\lra 
H^4_{{\rm et,c}}({Z^{{\rm sm}}}_{\bQ},\Q_\ell)\lra H^4_{{\rm et,c}}(\widetilde{Z}_{\bQ},\Q_\ell)=
H^4_{{\rm et}}(\widetilde{Z}_{\bQ},\Q_\ell)\lra $$
$$ \lra H^4_{{\rm et,c}}((\widetilde{L}_{0,1}\coprod \widetilde{L}_{2,3})_{\bQ},\Q_\ell)\lra 
H^5_{{\rm et,c}}({Z^{{\rm sm}}}_{\bQ},\Q_\ell)=0.$$
The vanishing of the last cohomology is due to Corollary \ref{vanishing-et}.
One can see that $$H^3_{{\rm et,c}}((\widetilde{L}_{0,1}\coprod \widetilde{L}_{2,3})_{\bQ},\Q_\ell)\simeq 
H^3_{{\rm et}}(F_{\bQ},\Q_\ell)^{\oplus 2}=0$$ and 
$H^4_{{\rm et,c}}((\widetilde{L}_{0,1}\coprod \widetilde{L}_{2,3})_{\bQ},\Q_\ell)\simeq 
H^4_{{\rm et}}(F_{\bQ},\Q_\ell)^{\oplus 2}=\Q_\ell(-2)^{\oplus 2}$. 
By Lemma \ref{vanishing-et}, one has 
$$H^4_{{\rm et,c}}({Z^{{\rm sm}}}_{\bQ},\Q_\ell)\simeq H^4_{{\rm et}}(\widetilde{Z}_{\bQ},\Q_\ell)^{\rm ss}/\Q_\ell(-2)^{\oplus 2}.$$
Put $D'=Z^{{\rm sm}}\setminus Z^\circ$. Then $D'$ consists of curves by Proposition \ref{line2}. 
The claim follows from the exact sequence 
$H^3_{{\rm et,c}}(D'_{\bQ},\Q_\ell)\lra 
H^4_{{\rm et,c}}({Z^\circ}_{\bQ},\Q_\ell)\lra H^4_{{\rm et,c}}({Z^{\rm sm}}_{\bQ},\Q_\ell)\lra 
 H^4_{{\rm et,c}}(D'_{\bQ},\Q_\ell)$ with $H^i_{{\rm et,c}}(D'_{\bQ},\Q_\ell)=0$ for $i>2$. 
\end{proof}
Putting (\ref{decomp}), Proposition \ref{main-coh}, and Theorem \ref{coh-deg2} together, one has 
\begin{thm}\label{main-theorem2} 
$L(s,H^2_{{\rm et}}(Z^\circ_{\bQ},\Q_\ell))=\zeta(s-1)^{8}L(s-1,\chi_{-1})^7
L(s-1,\chi_{2})^2 L(s-1,\chi_{-2})^2 L(s,g)$.
\end{thm} 

\subsection{Differential forms on $Z^\circ$}\label{diff}
In this subsection, we  shall discuss the relation of differential forms on $Z^\circ$ and some Siegel modular form 
of weight $(3,1)$. Throughout this subsection, we freely use the terminology of \cite{o-s}. 
Note that we are in position to apply the results of \cite{o-s} because $W$ is a smooth compactification of $Z^\circ$ 
so that the boundary is a normal crossing divisor. 

Since $\G_Z\backslash\H_2$  
is an Eilenberg-MacLane space of $\G_Z$, 
we have  $H^3(\G_Z,\C)=H^3(\G_Z\backslash\H_2,\C)=H^3(Z^\circ,\C)$. 
Put $j:Z^\circ\hookrightarrow W$. Then the parabolic cohomology of $\G_Z\backslash\H_2$ is 
defined by using Borel-Serre compactification and by Theorem 5.6 of \cite{o-s}, one has 
$$H^3_!(Z^\circ,\C)={\rm Im}(H^3_c(Z^\circ,\C)\longrightarrow H^3(Z^\circ,\C))={\rm Im}(
H^3(W,\C)\stackrel{j^\ast}{\longrightarrow}H^3(Z^\circ,\C)).$$
Since the LHS is independent of the choice of a smooth compactification of $Z^\circ$, the identification 
$H^3_!(Z^\circ,\C)={\rm Im}(
H^3(\widetilde{Z},\C)\stackrel{j'^\ast}{\longrightarrow}H^3(Z^\circ,\C))$ holds for the natural inclusion $j':Z^\circ\hookrightarrow \widetilde{Z}$ 
and therefore $H^3_!(Z^\circ,\C)=0$ since 
$H^3(\widetilde{Z},\C)=0$. 
It is well-known (cf. Section 7 in \cite{o-s}) that  
$H^3(Z^\circ,\C)=H^3_!(Z^\circ,\C)\oplus H^3_{{\rm Eis}}(Z^\circ,\C)=H^3_{{\rm Eis}}(Z^\circ,\C)$. 
Then via Eichler-Shimura embedding and the Poincar\'e residue map, one has 
$$M^2_3(\G_Z)^{{\rm KE}}\stackrel{ES}{\hookrightarrow} H^3(Z^\circ,\C)=H^3_{{\rm Eis}}(Z^\circ,\C)
\stackrel{{\rm res[1]}}{\lra}H^{2}(\widetilde{D}^{[1]},\C)\stackrel{P}{\lra} \bigoplus_{i=0,2}H^{2}(\widetilde{L}_{i,i+1},\C)$$
where the map $P$ is the natural projection with respect to the Hodge decomposition 
$H^{2}(\widetilde{D}^{[1]},\C)=\C^{\oplus 40}\oplus \bigoplus_{i=0,2}H^{2}(\widetilde{L}_{i,i+1},\C)$. 
Let us explain Eichler-Shimura map as follows. 
By Corollaire (3.2.13)-(ii) of \cite{deligne-hodge}, one has a natural isomorphism   
\begin{equation}\label{hodge}
H^0(W,\Omega^3(\log\partial W))\stackrel{\sim}{\lra}F^3H^3(Z^\circ,\C)\subset H^3(Z^\circ,\C)
\end{equation}
where $F^\ast$ is the decreasing filtration on $H^3(Z^\circ,\C)$. 
For $F\in M^2_3(\G_Z)^{{\rm KE}}$, the (logarithmic) differential 3-form $F(
\begin{bmatrix}
\tau_1 & \tau_2 \\
\tau_2 & \tau_3
\end{bmatrix}
)d\tau_1\wedge d\tau_2\wedge d\tau_3$ extends uniquely to $W$ and hence gives an element of $H^0(W,\Omega^3(\log\partial W))$. 
Combining this with (\ref{hodge}), 
one obtains $ES(F)\in H^3(Z^\circ,\C)$.  
The Poincar\'e residue map res$[1]$ is injective since ${\rm Ker}({\rm res}[1])=W_3H^3(Z^\circ,\C)=
H^3_!(Z^\circ,\C)=0$ (see p.492, (8) of \cite{o-s} ). 
By p.486, the last part of Section 1 of \cite{o-s}, one has $H^{2,0}(\widetilde{L}_{i,i+1},\C)= S^1_3(\G^1(4))$ and further 
$S^1_3(\G^1(4))\stackrel{\sim}{\lra}S^1_3(\Gamma^1_0(16),\chi_{-1})$ where the isomorphism is given by 
$f(\tau)\mapsto f(4\tau)$. 
One can see that the composite map of $ES$, ${\rm res}[1]$, and $P$ injects into $ \bigoplus_{i=0,2}H^{2}(\widetilde{L}_{i,i+1},\C)$. 

On the other hand, the natural inclusion map $\partial W\hookrightarrow W$ induces 
$H^{2,0}(W)\hookrightarrow H^{2,0}(\partial W)=H^2(\widetilde{L}_{0,1},\C)\oplus 
H^2(\widetilde{L}_{2,3},\C)$ (note that the injectivity is not the case and it strongly depends on our situation) 
and by Satz 6 of \cite{weissauer-vector}, one has a natural identification 
\begin{equation}\label{isom}
H^{2,0}(W,\C)=H^{2,0}(Z^\circ,\C)\simeq M^2_{(3,1)}(\G_Z).
\end{equation}
To make what we carry out more precise, we need to mention the relation between Satake compactification and 
the Siegel $\Phi$-operator. Recall that for a function $F$ on $\mathbb{H}_2$, the Siegel $\Phi$-operator is 
defined by 
$$\Phi(F)(\tau_1):=\lim_{t\to \infty}F
(\begin{bmatrix}
 \tau_1 & 0\\
0 & \sqrt{-1}t
\end{bmatrix}),\ \tau_1\in \mathbb{H}_1.$$ 
The coordinates $Y_3,Y_0,Y_2,Y_1,X_0,X_1,X_2,X_3$ correspond to 
$$\theta_{(1,0,0,0)}(\tau),\ \theta_{(1,1,0,0)}(\tau),\ \theta_{(1,0,0,1)}(\tau),\ \theta_{(1,1,1,1)}(\tau),$$
$$\theta_{(0,0,0,0)}(2\tau),\ \theta_{(0,1,0,0)}(2\tau),\ \theta_{(1,0,0,0)}(2\tau),\ 
\theta_{(1,1,0,0)}(2\tau),\ \tau\in \mathbb{H}_2 $$
respectively as we explained already at Remark \ref{Rem:defG_Z}. 

Put $s=
\begin{bmatrix}
0 & 1 \\
1 & 0
\end{bmatrix}$ and $s'=
\begin{bmatrix}
0 & 0 \\
2 & 0 
\end{bmatrix}$.  Define $g_0,g_2\in {\rm Sp}_2(\Q)$ by 
\begin{eqnarray}
g_0=
\begin{bmatrix}
s & 0_2 \\
0_2 & s
\end{bmatrix},\ 
g_2=
\begin{bmatrix}
s &  0_2 \\
s' & s
\end{bmatrix}. \label{eqn:defg0g2}
\end{eqnarray}
Since $\widetilde{L}_{i,i+1}$ is a part of the boundary of the Satake compactification $Z$, 
all of theta functions corresponding to the coordinates $Y_0,Y_1,Y_2,Y_3,X_i,X_{i+1}$ should be zero under 
the Siegel $\Phi$-operator. However it is easily checked that $X_i$ and $X_{i+1}$ never be zero under $\Phi$ simultaneously. 
This means that the usual Siegel $\Phi$-operator does not directly  relate to $\widetilde{L}_{i,i+1}$. 
This is compatible with the fact that $\Phi(F_Z)=0$. 
So we have to modify this. Fortunately one can easily see that 
$$\Phi(Y_0|[g_i])=\Phi(Y_1|[g_i])=\Phi(Y_2|[g_i])=\Phi(Y_3|[g_i])=\Phi(X_i|[g_i])=\Phi(X_{i+1}|[g_i])=0\ {\rm for}\ i=0,2$$
(see p.470 of \cite{ST} and use Lemma \ref{lem:igusa}) and by definition of Satake compactification, 
the pullback $g^\ast_{i}(\widetilde{L}_{i,i+1})$ is also a part of the boundary of $Z$. 
This is compatibles with the fact that 
$$\Phi(F_Z|[g_0])(\tau_1)=\Phi(F_Z|[g_0])(\tau_1)=g(\frac{\tau_1}{4})=\theta^2_{(0,0)}(\tau_1)
\theta^2_{(0,1)}(\tau_1)\theta^2_{(1,0)}(\tau_1)$$
where $g\in S_3(\G^1_0(16),\chi_{-1})$ is our elliptic newform. 
The pullback induces the isomorphism $H^2(g^\ast_i(\widetilde{L}_{i,i+1}),\C)\simeq H^2(\widetilde{L}_{i,i+1},\C)$. 
This gives a natural injection 
\begin{equation}\label{isom1}
H^{2,0}(Z^\circ,\C)\simeq H^{2,0}(W,\C)\stackrel{\hookrightarrow}{\lra}
\ds\bigoplus_{i=0,2}H^2(\widetilde{L}_{i,i+1},\C)\simeq \bigoplus_{i=0,2}H^2(g^\ast_i(\widetilde{L}_{i,i+1}),\C)
\end{equation}

Summing up, one has the following commutative diagram:
$$
\begin{CD}
H^3(\G_Z\backslash\H_2,\C)=H^3_{{\rm Eis}}(\G_Z\backslash\H_2,\C) @> {P\circ\rm res[1]} >> 
\bigoplus_{i=0,2}H^2(g^\ast_i(\widetilde{L}_{i,i+1}),\C) @< \stackrel{(\ref{isom1})}{\hookleftarrow}<<  H^{2,0}(Z^\circ,\C)\\
ES @AAA  \hookrightarrow @AAA \stackrel{(\ref{isom})}{\simeq} @AAA    \\
M^2_3(\G_Z)^{{\rm KE}} @>\Phi_1 >> S^1_3(\G^1(4))\oplus S^1_3(\G^1(4)) @< \Phi_2 << M^2_{(3,1)}(\G_Z).
\end{CD}
$$
The bottom maps are defined by $\Phi_i(F)=(\Phi(F|[g_0]),\Phi(F|[g_2])), i=1,2$ where $\Phi$ is the 
Siegel operator (see \cite{arakawa} for the Siegel $\Phi$-operator in case vector valued Siegel modular forms).  
Summing up, we have shown the following:
\begin{prop}\label{dim} ${\rm dim}_\C M^2_3(\G_Z)^{{\rm KE}}=
{\rm dim}_\C M^2_{(3,1)}(\G_Z)={\rm dim}_\C  H^{2,0}(Z^\circ,\C)=1$.
\end{prop}

Recall our 6-tuple $F_Z$ of theta constants.
 It will be proved at next section that $F_Z$ is an element of $M^2_3(\G_Z)^{{\rm KE}}$, hence 
giving a generator. 
Then by above commutative diagram, it is quite natural to find out the generator $E_Z$ of $ M^2_{(3,1)}(\G_Z)$ 
so that $\Phi_1(F_Z)=\Phi_2(E_Z)$. This study will be a main topic of next section. 

\section{Siegel modular forms on $Z^\circ=\G_Z\backslash \mathbb{H}_2$}\label{construction}
Let $\A$ be the ring of adeles of $\Q$ and $\widehat{\Z}$ be the profinite completion of $\Z$. 
For a commutative ring $R$, let 
\[
\GSp_n(R) = \l\{g \in \GL_{2n}(R)\ \Bigg|\ g \begin{bmatrix}
0_n & -1_n\\
1_n & 0_n  
\end{bmatrix}{}^tg = \nu(g) \begin{bmatrix}
0_n & -1_n\\
1_n & 0_n  
\end{bmatrix}\ {\rm for\ some\ } \nu(g)\in R^\times \r\}
\]
where $\nu$ is the similitude character given by $g \mapsto \nu(g)$. 
Let $\Sp_n(R) ={\rm Ker}(\nu)$. For $t_i \in R^\t$, let ${\rm diag}(t_1, \ldots, t_n)$ denote the diagonal matrix of size $n$. 

Suppose that $\eta$ is a Dirichlet character $\eta$ defined modulo $N$.
We obtain from $\eta$ the automorphic character $\ds\ot_p \eta_p$ on $\A^\t$ by the rule $\eta_p(p) = \eta(p)^{-1}$ for all $p \nmid N$via the class field theory and denote this character by the same symbol $\eta$.
For a congruence subgroup $\G $ of $\Sp_n(\Z)$, let $\G_p$ denote the completion of $\G$ at a nonarchimedean place $p$, and $\G_\A = \prod_{p<\infty} \G_p$.
Denote by $\G^n(N)$ (resp. $\G^n_0(N)$) the subgroup of ${\rm Sp}_n(\Z)$ consisting of the elements 
$g=\begin{bmatrix}
A_g & B_g\\
C_g & D_g  
\end{bmatrix}$ so that $g\equiv 1_{2n}\ {\rm mod}\ N$ (resp. $C_g \equiv 0_n\ {\rm mod}\ N$). 
For a character $\xi$ of $\G/\G^n(N)$, define $\xi_p$ on $\G_p$ via $\xi_p(u_p) = \xi^{-1}(u)$ by using an element $u \in \G$ such that $u \equiv u_p \pmod{NM_{2n}(\Z_p)}$ and $u \equiv 1 \pmod{NM_{2n}(\Z_v)}$ for $v \neq p$.
Define 
\begin{eqnarray}
A_n(\Z_p) &=& \{a_n(t) := \begin{bmatrix}
I_n & 0_n \\
0_n & tI_n 
\end{bmatrix} \mid t \in \Z_p^\t \}, \label{eqn:defan} \\
L_n(\Z_p) &=& \{
l_n(t_1,\ldots,t_n) := {\rm diag}(t_1,\ldots,t_n,t_1^{-1}, \ldots, t_n^{-1})\mid t_i \in \Z_p^\t\}. \label{eqn:defln}
\end{eqnarray} 
We denote by $\wt{\G}_p$ the compact subgroup of $\GSp_n(\Q_p)$ generated by $\G_p$ and $A_n(\Z_p)$, and define $\wt{\G}_\A = \prod_{p < \i} \wt{\G}_p$.
Let $(\rho,V_\rho)$ be an algebraic irreducible representation of $\GL_n(\C)$. 
Let $\mathbb{H}_n$ be the Siegel upper half space of degree $n$. 
Let an element $\gamma= 
\begin{bmatrix}
A & B\\
C & D  
\end{bmatrix}\in \GSp_n(\R)$ act on $\tau \in \mathbb{H}_n$ by 
$\gamma \tau := (A\tau+B)(C\tau+D)^{-1}$, and $F|[\g]_\rho(\tau) := \rho^{-1}(J(\gamma,\tau))F(\gamma \tau)$ for a $V_\rho$-valued function $F$ on $\mathbb{H}_n$, where $J(\gamma,\tau)=C\tau+D$. 
For a congruence subgroup $\G$ with a finite character $\xi$ on $\G$, we denote 
by $M^n_\rho(\G,\xi)$ the space of $V_\rho$-valued holomorphic functions $F$ satisfying the condition: $$F|[\gamma]_\rho =\xi(\gamma)F, \g\in\G$$
and if $n=1$, we further impose the condition: 
$$\lim_{\tau\to \sqrt{-1}\infty}F(\gamma \tau)<\infty\ {\rm for\ any}\ \gamma\in {\rm Sp}_1(\Q)={\rm SL}_2(\Q).$$ 
We sometimes drop $\xi$ when $\xi$ is trivial and $n$ when $n$ is clear from the context.
For a pair of positive integers $(k_1,k_2)$ so that $k_1\ge k_2$, we define the 
algebraic representation $\rho_{k_1,k_2}$ of $GL_2$ by 
$V_{k_1,k_2}={\rm det}^{k_2} {\rm St}_2\otimes{\rm Sym}^{k_1-k_2}{\rm St}_2$  
where ${\rm St}_2$ is the standard representation of dimension 2.
We write $M^n_k(\G,\xi)$ (resp. $M^2_{(k_1,k_2)}(\G,\xi)$), if $\rho=\det^k$ (resp. $n=2$ and $\rho=\rho_{k_1,k_2}$).  
By the strong approximation theorem for $\Sp_n$, one can write any element of $\Sp_n(\A)$ as $g =\g g_\i u$ with $\g \in \Sp_n(\Q), g_\i \in \Sp_n(\R), u \in \G_\A$.
Using this expression, one obtains from $F \in M_\rho^n(\G,\xi)$ the automorphic form $F^\sharp$ on $\Sp_n(\A)$ by 
\begin{eqnarray*}
F^\sharp(g) =F^\sharp(\g  g_\i u) = F^\sharp(g_\i u) =  F(g_\i \ir_{n}) \rho(J(g_\i,\ir_{n})) \prod_p \xi_p(u_p),
\end{eqnarray*}
where $\ir_{n} = \sqrt{-1}I_n \in \H_{n}$.
We call $F^\sharp$ the {\it standard extension} of $F$.
Further, one can express any element of $\GSp_n(\A)$ as $g =z \g g a_n(t)$ with $\g \in \GSp_n(\Q), z \in \R^\t_+, g \in \Sp_n(\A)$ and $t \in \widehat{\Z}^\t$.
Let us denote by $(c_1,\ldots,c_n)$ the highest weight of $\rho$ with $c_1 \ge \ldots, \ge c_n$, and put 
$$c(\rho) = \Big(\sum_{i=1}^n c_i\Big) -\frac{n(n+1)}{2}.$$
We extend $F^\sharp$ to the function $\wt{F}$ on $\GSp_n(\A)$ by 
\begin{eqnarray}
\wt{F}(g)= z^{c(\rho)}F^\sharp(g'), \label{eqn:defstdext}
\end{eqnarray}
according to the cohomological interpretation of Siegel modular forms and the geometric Hecke operators (cf. \cite{ichikawa}).
We also call $\wt{F}$ the standard extension of $F$ or that of $F^\sharp$. 
The central character $w_{\wt{F}} = \ot_p w_p$ of $\wt{F}$ has weight $c(\rho)$ and determined by the values $w_{p}(t) = \xi_p(l(t,\ldots,t))$ for $t \in \Z_p^\t$ at all bad places $p$. 
We define 
\begin{eqnarray*}
\G_d^n(N) = \G_d(N) := \{(m_{ij}) \in \Sp_n(\Z) \mid m_{ij} \equiv 0 \pmod{N}, \mbox{if $i \neq j$} \}
\end{eqnarray*}
and decompose 
\begin{eqnarray*}
M_\rho^{n}(\G(N)) = \bigoplus_\d M_\rho^{n}(\G_d(N),\d) 
\end{eqnarray*}
where $\d = (\d_1,\ldots, \d_n)$ sending $l(t_1,\ldots,t_n)$ to $\prod_i \d(t_i)$ runs all characters over $\G_d^n(N)/\G^n(N) \simeq ((\Z/N\Z)^\t)^n$.
Let $p$ be a prime such that $p \nmid N$.
Let $\Delta_p = \{l_n(p^{t_1}, \ldots, p^{t_n}) a_n(p^{t}) \mid t,t_i \in \Z  \}$.
For $h \in \Delta_p$, decompose $\G(N) h \G(N) = \sqcup_j \G(N) h_j$.
Then, each $h_j$ is congruent to $h$ modulo $N$.
Following Evdokimov \cite{Ev}, we define a classical Hecke operator
\begin{eqnarray*}
T_\i(h)F := \sum_j F|[h_j]_\rho
\end{eqnarray*}
for $F \in M_\rho(\G_d(N),\d)$.
Let $e_v$ denote the usual embedding $\GSp_n(\Q) \to \GSp_n(\Q_v)$.
Then, it holds that $\GSp_n(\Z_p) e_p(h) \GSp_n(\Z_p) = \sqcup_j e_p(h_j) \GSp_n(\Z_p)$.
We define the local Hecke operator $T_p(h) \wt{F}(g) = \sum_j \wt{F}(ge_p(h_j))$.
If $\wt{F}$ belongs to an unramified irreducible representation of $\GSp_n(\Q_p)$, then $\wt{F}$ is an eigen function for $T_p(h)$.
These $T_\i(h), T_p(h)$ are endomorphisms of $M_\rho(\G_d(N),\d)$.  
\begin{prop}\label{prop:egvgl}
Let $F \in M_\rho^{n}(\G_d(N),\d)$.
The standard extension $\wt{F}$ generates an unramified irreducible representation of $\GSp_n(\Q_p)$ if and only if $F$ is an eigenform for $T_\i(h)$ for any $h = l_n(p^{t_1}, \ldots, p^{t_n}) a_n(p^t) \in \Delta_p$ with $p \nmid N$.
Further it holds that
\begin{eqnarray*}
\lambda_p^\i = \lambda_p \d\l(l_n(p^{t_1}, \ldots, p^{t_n})\r),
\end{eqnarray*} 
where $\lambda_p^\i$ (resp. $\lambda_p$) is the eigenvalue of $T_\i(h)$ (resp.$T_p(h)$). Here we identify $l_n(p^{t_1}, \ldots, p^{t_n})$ with an element of $\G_d(N)/\G(N)$.
\end{prop}
\begin{proof}
Using the relation (\ref{eqn:defstdext}), and the left $\GSp_n(\Q)$-invariance property of $\wt{F}$, we have
\begin{eqnarray*}
T_\i(h) F(z) 
&=& \sum_j F|[h_j]_{\rho}(z) = \sum_j \wt{F}(e_\i(h_j) g_\i) \rho(g_\i, \ir_n)\\
&=& \sum_j \wt{F}\l(g_\i \prod_{v < \i}e_v(h_j^{-1})\r) \rho(g_\i, \ir_n),
\end{eqnarray*}
where $z = g_\i {\mathrm i}_n$.
Since $\wt{F}$ is right $\wt{\G(N)}_v$-invariant for all $v < \i$, 
\begin{eqnarray*}
T_\i(h) F(z) 
&=& \sum_j \wt{F}\l(g_\i e_p(h_j^{-1}) \prod_{v \mid N}e_v(h_j^{-1})\r) \rho(g_\i, \ir_n)\\
&=& \d(l_n(p^{t_1}, \ldots, p^{t_n})) \sum_j  \wt{F}\l(g_\i e_p(h_j^{-1}) \r) \rho(g_\i, \ir_n)\\
&=& \lambda_p \d\l(l_n(p^{t_1}, \ldots, p^{t_n})\r) \wt{F}(g_\i) \rho(g_\i, \ir_n) = \lambda_p \d\l(l_n(p^{t_1}, \ldots, p^{t_n})\r) F(z). 
\end{eqnarray*}
Now the assertion follows immediately.
\end{proof} 
Define $T_p(p^n)$ by the sum of all double cosets $\GSp_n(\Z_p) h \GSp_n(\Z_p)$ such that $\nu(h) = p^n$ (Note that all $h$ can be taken from $\Delta_p$).
For an unramified irreducible representation of $\Pi_p$ of $\GSp_n(\Q_p)$, let  $\lambda(p^n)$ be 
the eigenvalue of a right $\GSp_n(\Z_p)$-fixed vector $f \in \Pi_p$.
Then, $\lambda(p^n)$ is independent from the choice of $f \in \Pi_p$ and the denominator polynomial $Q_n(p^{-s})$ of the Dirichlet series 
$\ds\sum_{n=0}^\i \lambda(p^n) p^{-ns}$ is of degree $2n$.
We call $Q_n(p^{-s})$ the spinor $L$-function of $\Pi_p$ and denote by $L(s,\Pi_p;{\rm spin})$.

For the case of $n=2$, following \cite{arakawa} and \cite{Ev}, we define the $L$-function of $F \in M_{k_1,k_2}(\G(N),\d)$ such that $T_\i(a_2(p^i))F = \lambda_\i(a_2(p^i)) F$ for $i = 1,2$ by
\begin{eqnarray*}
\Big(1- \lambda_\i(a_2(p)) p^{-s} + \l(\lambda_\i(a_2(p))^2 - \lambda_\i(a_2(p^2)) -\d^{-1}(l_2(p,p)) p^{\mu_\rho-1}\r) p^{- 2s} \\
- \d^{-1}(l_2(p,p))\lambda_\i(a_2(p)) p^{\mu_\rho-3s} + \d^{-1}(l_2(p,p))^2 p^{2\mu_\rho-4s}\Big)^{-1}
\end{eqnarray*}
with $\mu_\rho = k_1+k_2-3$.
We denote this $L$-function by $L(s,F;{\rm AE})_p$.
By using Proposition \ref{prop:egvgl}, one finds
\begin{eqnarray}\label{AE}
L(s,F; {\rm AE})_p = L(s,\pi_{p}^\vee; {\rm spin}) \label{eqn:ANLSPL}
\end{eqnarray}
where $\pi_p^\vee = \pi_p \ot \w_{\pi_p}^{-1}$ indicates the contragradient of the unramified irreducible representation $\pi_p$ generated by $\wt{F}$.
Similarly, for an elliptic modular form $f$ which is eigenform at $p$, $L(s,f)_p$ of the classical $L$-function of $f$ coincides with $L(s,\pi_{\wt{f},p}^\vee)$ of $\pi_{\wt{f}}$ generated by $\wt{f}$.  

We fix a standard additive character $\psi = \ot_v \psi_v$ on $\Q \bs \A$ by $\psi_\i(z) = \exp(2\pi \sqrt{-1} z)$ for $z \in \R$ and $\psi_p(z) = \exp(-2\pi \sqrt{-1} {\rm Fr}(z))$, where ${\rm Fr}(z)$ indicates the fractional part of $z \in \Q_p$.
For $a \in \Q$, define $\psi_a$ by $\psi_a(z) = \psi(a z)$.
Let $\SS_2(R) = \{m ={}^t m \in M_2(R)\}$ for an algebra $R$.
For an automorphic form $F$ on $\Sp_2(\A)$, the Fourier coefficient $F_T$ of $F$ associated to $T  \in \SS_2(\Q)$ is defined by 
\[
F_T(g) = {\rm vol}(\SS_2(\Q) \bs \SS_2(\A))^{-1}\int_{\SS_2(\Q) \bs \SS_2(\A)} \psi({\rm tr}(sT))^{-1}F\l(\begin{bmatrix}
1_2 &s \\
0_2 & 1_2
\end{bmatrix}g\r) ds.
\]
{\bf A)} Suppose ${\rm rank} (T) =2$.
Let $\GO_T = \{g \in \GL_2 \mid  {}^tg Tg = \nu(g) T\}$ and $\O_T$ the kernel of the similtude character $\nu$.
Let $\GSO_T(\Q_v)$ denote the kernel of $\det^{-1}\cdot \nu$ and let $\SO_T(\Q_v) = \GSO_T(\Q_v) \cap \O_T(\Q_v)$.
Then, $\GSO_T(\Q_v), \SO_T(\Q_v)$ are commutative.
For an automorphic quasi-character $\xi$ on $\SO_T(\A)$, the Bessel function $F_T^\xi$ is defined by 
\[
F_T^\xi(g) = {\rm vol}(\SO_T(\Q) \bs \SO_T(\A))^{-1}\int_{\SO_T(\Q) \bs \SO_T(\A)} \xi(k)^{-1} F_T\l(\begin{bmatrix}
k & 0_2 \\
0_2 & {}^t k^{-1}  
\end{bmatrix}g\r) d k.
\] 
Since $\SO_T(\Q) \bs \SO_T(\A)$ is compact, this integral converges.\\
{\bf B)}Suppose ${\rm rank} (T) = 1$.
Then $\GSO_T(\Q)$ is isomorphic to the Borel parabolic subgroup $B(\Q)$ of $\GL_2(\Q)$, and $\SO_T(\Q)$ is isomorphic to the unipotent radicals of $B(\Q)$.
For the additive character $\psi_a$ on $\Q \bs \A$, the Bessel function $F_T^{\psi_{a}}$ is defined by 
\[
F_T^{\psi_a}(g) = {\rm vol}(\Q \bs \A)^{-1} \int_{\Q \bs \A} \psi_a(b)^{-1} F_T\l(\begin{bmatrix}
n_b & 0_2 \\
0_2 & {}^tn_b^{-1}  
\end{bmatrix}g\r) d b,
\] 
where $n_b$ indicates the element of $\GL_2(\A)$ correponding to $\begin{bmatrix}
1 & b\\
 & 1
\end{bmatrix}$ via the above isomorphism.
Therefore, from the Fourier expansion of $F$, we obtain a closer expansion:
\begin{eqnarray}
F = F_0 +  \sum_{T:{\rm rank}(T) =2} \sum_\xi F_T^\xi + \sum_{T:{\rm rank}(T) =1}\sum_{a \in \Q} F_T^{\psi_a}. \label{eqn:expF} 
\end{eqnarray}
\begin{Rem}\label{Rem:cuspidality}
i)
All of $F_T^{\psi_0}$ and $F_0$ vanish, if and only if $F$ is a cusp form.\\
ii)
Let $T_\a = \begin{bmatrix}
0 & 0\\
0 & \a
\end{bmatrix}$ with $\a \in \Q$.
For $g \in \GL_2(\A)$, $E_{T_\a}^{\psi_0}(1,g)$ coincides with the $\a$-th Fourier coefficient of the elliptic modular form $\Phi(E)$, and therefore $\Phi(E)(g) = \sum_\a E_{T_\a}^{\psi_0}(1,g)$, where $(1,g)$ indicates the image of $(1,g)$ by the embedding (\ref{eqn:emb1}).\\
iii) $F_T^\xi$ and $F_T^{\psi_a}$ are decomposable: $F_T^\xi = \ot_v F_{T,v}^\xi, F_T^{\psi_a} = \ot_v F_{T,v}^{\psi_a}$.
\end{Rem}
Let us recall some results of Sally and Tadi\'c \cite{Sa-T} on parabolic induced representations of $\GSp_2$ over nonarchimedean local field.
Let $Q$ be the Klingen parabolic subgroup with Levi decomposition $Q = N_QM_Q$, and $B$ be the Borel one with $B = N_BM_B$. 
We make the following identifications $\GL_1 \t \GL_2$ (resp. $\GL_1 \t \GL_1 \t \GL_1$) and $M_Q$ (resp. $M_B$):
\begin{eqnarray}
\l(t,\begin{bmatrix}
a &b \\
c &d  
\end{bmatrix}\r) \longrightarrow \begin{bmatrix}
t &0 &0 &0 \\
0 & a&0 &b \\
0 &0 &  \ds\frac{ad-bc}{t} &0 \\
0 & c& 0&d 
\end{bmatrix} \in M_Q \subset Q = \l\{
\begin{bmatrix}
* & *&* &* \\
0 &* &* &*\\
0 &0 &* &0 \\
0 &* &* &* 
\end{bmatrix} \in \GSp_2 \r\} \label{eqn:emb1}
\end{eqnarray}
(resp.
\begin{eqnarray*}
(t,a,d) \longrightarrow \begin{bmatrix}
t &0 &0 &0 \\
 0& a&0 &0 \\
 0&0 &  \frac{d}{t} &0 \\
0 &0 &0 & \frac{d}{a} 
\end{bmatrix} \in M_B \subset B = \l\{
\begin{bmatrix}
* & *&* &* \\
0 &* &* &*\\
0 &0 &* &0 \\
0 &0 &* &* 
\end{bmatrix} \in \GSp_2 \r\}).
\end{eqnarray*}
Let $\F$ be a local field.
From a pair of quasi-character $\xi$ of $\F^\t$ and irreducible admissible representation $\pi$ of $\GL_2(\F)$ (resp. a triple of quasi-characters $\xi_1,\xi_2,\xi_3$ of $\F^\t$), we obtain a representation $\xi \ot \pi$ sending $M_Q \ni (t,g) \to \xi(t) \pi(g)$ (resp. $\xi_1 \ot \xi_2 \ot \xi_3$ sending $M_B \ni (t,a,d) \to \xi_1(t) \xi_2(a) \xi_3(d)$), and extend it to $Q= N_QM_Q$ (resp. $B = N_BM_B$) trivially. 
We call the parabolic induced representation to $\GSp_2(\F)$ from this representation {\bf the local Klingen (resp. Borel) parabolic induction associated to $\xi, \pi$ (resp. $\xi_1,\xi_2,\xi_3$)}, and denote it by $\xi \rt \pi$ (resp. $\xi_1 \t \xi_2 \rt \xi_3$).
If $\chi$ is a quadratic character of $\F^\t$ and $\pi$ is supercuspidal, then $|*|_p \chi \t \chi \rt |*|_p^{-1/2}\xi$ has four irreducible constituents, and $|*|_p \chi \rt |*|_p^{-1/2}\pi$ has two irreducible constituents, generic $\d(|*|_p\chi,|*|_p^{-1/2}\pi)$ and non-generic $L(|*|_p\chi,|*|_p^{-1/2}\pi)$.
\subsection{theta constants.}
Let $\m =(m_1,\ldots, m_{2n}) \in \R^n$, and write $\m = (\m',\m'')$ with $\m',\m'' \in \R^n$.
The so-called Igusa theta constant for $\m$ ($\m$ is called a characteristic) is defined by
\begin{eqnarray*}
\th_{\m}(\tau) = \sum_{a \in \Z^n} \exp\l(\pi \sqrt{-1}\l( \tau\l[a+\frac{\m'}{2}\r] + \l(a+\frac{\m'}{2}\r){}^t\m''\r) \r),
\end{eqnarray*}
where $\tau \in \H_n$ and $\tau[x] = x\tau {}^tx$.
Igusa theta constants are Siegel modular forms of weight $\tha$ for some congruence subgroup.
In the case that all of entries of $\m$ are integers, $\th_{\m}(\tau)$ is determined by the values of $m_i$ modulo $2$, and $\th_{\m}$ (resp. $\m$) is called an {\bf even} or {\bf odd} theta constant (resp. even or odd characteristic), according to $\m'{}^t\m'' \equiv \pm 1 \pmod{2}$.  
Although odd theta constant itself is vanishing (cf. p. 226 of \cite{Ig}), `local' odd theta constant is not vanishing in the sense as explained below.
Let $V_{2r} = (V_{2r}, Q_{2r})$ be the $2r$-dimensional anisotropic quadratic space over $\Q$ with $Q_{2r}(x,y) = {}^tx y$.   
Obviously, 
\begin{eqnarray*}
V_{2r} \simeq \underbrace{V_2 \perp \cdots \perp V_2}_{r}. \ \  \mbox{(isometric)}
\end{eqnarray*}
We embed $\Sp_n \t \O_{V_{2r}}$ into $\Sp_{2nr}(\Q_v)$ in the usual way, and obtain the Weil representation $\w_v^{n,2r}(g,h)$ of $(g,h) \in \Sp_n \t \O_{V_{2r}}$ on $\Sc(V_{2r}(\Q_v)^n)$ associated to $\psi_v$ such that
\begin{eqnarray}
\w_v^{n,2r}(1,h)\vp(x) &=& \vp(h^{-1}x), \label{eqn:weilprop1}\\
\w_v^{n,2r}\Big(\left[\begin{array}{cc}
a & 0 \\
0 & {}^ta^{-1}
\end{array}\right],1\Big)\vp(x) &=& \chi_{-1}(\det a)^r|\det a|_v^{r}\vp(x a),\label{eqn:weilprop2} \\
\w_v^{n,2r}\Big(\left[\begin{array}{cc}
I_n & b \\
 0& I_n
\end{array}\right],1\Big)\vp(x) &=& \psi_v\Big(\frac{{\rm tr}(b Q_{2r}(x,x))}{2}\Big)\vp(x), \label{eqn:weilprop3}
\end{eqnarray}
Then, we can regard each $2r$-tuple product $\th(\tau)$ of theta constants $\th_{\m_1}(\tau), \ldots \th_{\m_{2r}}(\tau)$ as the classical form of the automorphic form 
\begin{eqnarray}
\vartheta_{\m_1, \ldots, \m_{2r}}(g) = \sum_{x \in V_{2r}(\Q)^n} \l(\prod_v \w_v^{n,2r}(g_v) \vp_v(x)\r) \label{eqn:thetamm'}
\end{eqnarray}
on $\Sp_n(\A)$ for $\vp_\i(x_1,\ldots,x_n) = \exp\l(-2\pi ({}^tx_1x_1 + \cdots +{}^nx_nx_n)\r)$ and some $\vp_p$ belonging to $\Sc(V_{2r}(\Q_p)^n)$, the space of Schwartz-Bruhat functions on $V_{2r}(\Q_p)^n$.
In the remainder of this section, we will consider only the case that all entries of $\m_i$ are integers.
In this case, at (\ref{eqn:thetamm'}), 
\begin{eqnarray*}
\vp_p(x_1,\ldots,x_n) = \prod_{i=1}^n \Ch\l(x_i; V_{2r}(\Z_p)\r)
\end{eqnarray*}
for all odd prime $p$, where $\Ch$ indicates the characteristic function.
\begin{lem}[Igusa's theta transformation formula]\label{lem:igusa}
For an element $M = \begin{bmatrix}
A &B \\
C & D  
\end{bmatrix} \in \Sp_{2n}(\Z)$, and a characteristic $\m \in \Q^{2n}$, put $M \cdot \m = \m M^{-1} + \tha ({\rm diag}(C{}^tD), {\rm diag}(A{}^tB))$, and 
\begin{eqnarray*}
\phi_{\m}(M) = -\tha \l(\m' {}^tDB{}^t\m' - 2 \m'{}^tBC{}^t\m'' +\m''{}^tCA{}^t\m'' -(\m'{}^tD - \m''{}^tC){}^t{\rm diag}(A{}^tB)\r).
\end{eqnarray*} 
Then, for $\tau \in \H_{2n}$, it holds that
\begin{eqnarray*}
\th_{M\cdot \m}(M\tau) = \k(M)\exp(2\pi \ir \phi_\m(M))J(M,\tau)^{\ha}\th_\m(\tau),
\end{eqnarray*}
where $\k(M)$ is a root of unity depending on $M$ and the choice of the square root of $J(M,\tau)$.
In case of $\g \in \G(2)$, it holds that $\k(M)^2 = (-1)^{{\rm trace}(D-1)/2}$.
\end{lem}
By using Lemma \ref{lem:igusa}, from a $2r$-tuple product $\th_{\m}(\tau)$, one obtains a congruence character on $\G(2)$ by $\th_\m|\g/\th_\m$ which is trivial on $\G(4,8)$ (cf. section 5,6 of \cite{GS}).
We will denote this character by $\chi_\m$.
In the case that all entries of $\m$ are integral, Lemma \ref{lem:igusa} is able to be considered as the $\Sp_n(\Z_2)$ transformation formula for $\vartheta_{\m_1, \ldots, \m_{2r}}$, and is determined by the $2$-adic component of the Schwartz-Bruhat function.
We will observe this component in the case of $r=1, n = 2$, and denote it by $\vp_{\m_1,\m_2}$.
The finite group $\G(2)/\G(4,8)$ for $\G(2) \subset \Sp_2(\Z)$ is abelian, and generated by the following ten elements (cf. \cite{GS}):
\begin{eqnarray*}
e_1=\begin{bmatrix}
 1&0 &0 &0 \\
2 &1 &0 &0 \\
0 &0 & 1&-2 \\
0 &0 &0 &1 
\end{bmatrix},
e_3=\begin{bmatrix}
 1&0 &0 &2 \\
 0&1 &2 &0 \\
 0&0 & 1&0 \\
 0&0 & 0&1 
\end{bmatrix},
e_2={}^te_1,
e_4={}^te_3,
e_5=-I_4,\\
e_6= \begin{bmatrix}
-1&0 &0 &0 \\
0 & 1&0 &0 \\
0 &0 & -1 &0 \\
0 &0 &0 & 1
\end{bmatrix},
e_7 = 
\begin{bmatrix}
 1&0 &2 &0 \\
0 &1 &0 &0 \\
0 &0 & 1&0 \\
0 &0 &0 &1 
\end{bmatrix},
e_8= \begin{bmatrix}
1 & 0&0 &0 \\
0 &1 &0 &2 \\
0 &0 &1 &0 \\
0 & 0&0 &1 
\end{bmatrix},
e_9={}^te_7,e_{10}={}^te_8.
\end{eqnarray*}
In section 6 of \cite{GS}, using Igusa's transformation formula, van Geemen and van Straten obtained the table (see TABLE 1 below) of the values 
$\chi_{\m_1,\m_2}(e_i)$ in case that each $\m_j = (a_j,b_j,c_j,d_j)$ is even and $n = 2$.
\begin{table}[htbp]
\label{tab1}
\begin{center}
{\renewcommand\arraystretch{2}
\begin{tabular}{|c|c|c|c|c|c|}
\hline
\ $i$ \  & $1$ & $2$ & $3$ & $4$ & $5$   \\
\hline
$\chi_{\m_1,\m_2}(e_i)$ & $(-1)^{\sum b_jc_j}$ & $(-1)^{\sum a_jd_j}$ & $(-1)^{\sum a_jb_j}$ & $(-1)^{\sum c_jd_j}$ & $1$  \\
\hline
\ $i$ \   &$6$ & $7$ & $8$  & 9 & 10  \\
\hline
$\chi_{\m_1,\m_2}(e_i)$  & $(-1)^{1+\sum a_jc_j}$ & $\ir^{\sum a_j}$  & $\ir^{\sum b_j}$ & $\ir^{\sum c_j}$ & $\ir^{\sum d_j}$ \\
\hline
\end{tabular}}
\end{center}
\caption{}
\end{table}

\begin{prop}\label{prop:table}
The above TABLE 1 is still valid for $\chi_{\m_1,\m_2}$ even if  $\m_1$ or $\m_2$ is odd. 
\end{prop}
\begin{proof}
In case that $\m_i$ for $i = 1,2$ is odd (resp. even), take its extension $\n_i = \l((\m_i',1), (\m_i'',1)\r) \in \R^6$ (resp. $\n_i = \l((\m_i',0), (\m_i'',0)\r) \in \R^6$).
Then, this $\n_i$ is even, and therefore Igusa's transformation formula works for $\th_{\n_i}$.
We can write 
\begin{eqnarray*}
\vp_{\n_1,\n_2}(x_1,x_2,x_3) = \phi(x_1,x_2) \phi'(x_3)
\end{eqnarray*}
by some $\phi \in \Sc(V_2(\Q_2)^2)$ and $\phi' \in \Sc(V_2(\Q_2))$.
It holds that
\begin{eqnarray}
\vp_{\n_1,\n_2}(x_1,x_2,x_3)\w_2^{2,2}(g,1)\phi(x_1,x_2) = \phi(x_1,x_2)\w_2^{3,2}(i_2^3(g),1)\vp_{\n_1,\n_2}(x_1,x_2,x_3) \label{eqn:relth}
\end{eqnarray}
for $g \in \Sp_2(\Q_v)$, where $i_2^3$ indicates the embedding of $\Sp(2)$ into $\Sp(3)$:
\begin{eqnarray*}
g = \begin{bmatrix}
a & b\\
c & d 
\end{bmatrix}
\longrightarrow 
\begin{bmatrix}
a &0 &b &0 \\
 0& 1&0 &0 \\
c &0 &d &0 \\
 0& 0&0 & 1
\end{bmatrix}.
\end{eqnarray*}
It follows from (\ref{eqn:relth}) that
\begin{eqnarray*}
J(\g,\tau)\th_{\n_1}\th_{\n_2}\l(\begin{bmatrix}
\tau &0 \\
 0& \sqrt{-1} 
\end{bmatrix}\r) \w_2^{2,2}(\g,1) \vp_{\m_1,\m_2}= 
\th_{\n_1}\th_{\n_2}\l(\begin{bmatrix}
\g  \tau &0 \\
0 & \sqrt{-1} 
\end{bmatrix}\r)\vp_{\m_1,\m_2}
\end{eqnarray*} 
for $\g \in \G(2)$.  
Using this relation, and applying Lemma \ref{lem:igusa} to $\th_{\n_1}\th_{\n_2}$, one can verify the assertion.
\end{proof}
\subsection{Siegel threefold $Z$ and a six tuple product of theta constants.}
As mentioned in Section 1, van Geemen and Nygaard in \cite{vG-N} studied the Siegel threefold variety $\cA(2,4,8)$ defined by the Igusa group $\G(2,4,8)$ using a theta embedding to $\P^{13}$ via the $10$ even theta constants and $4$ ones twisted.
Then, the set of six tuple products of distinct even theta constants contains $210$ Siegel modular forms of weight $3$ for the Igusa group
\begin{eqnarray*}
\G(4,8) = \l\{\begin{bmatrix}
A & B\\
C & D  
\end{bmatrix} \in \G(4) (\subset \Sp_2(\Z)) \mid {\rm diag}(B) \equiv {\rm diag}(C) \equiv 0 \pmod{8} \r\}.
\end{eqnarray*}
Under the actions of $\Sp_2(\Z)$, this set breaks into three $\Sp_2(\Z)$-orbits $O_1, O_2, O_3$.
In \cite{vG-N}, they showed that the standard extension of each form in $O_1$ is a Hecke eigen cusp form outside $2$, and its spinor $L$-function coincides with 
$L(s,\pi_1 \ot \chi_{-1})L(s-1,\chi_{-1}) L(s-2,\chi_{-1})$, outside $2$, where $\pi_1$ is the unique irreducible automorphic cuspidal representation of $\GL_2(\A)$ generated by the unique newform of $S_4^1(\G^1_0(8))$.
It is possible to show that $O_1$ is contained in the $\chi_{-1}$-twist of a constituent of the global Saito-Kurokawa packet associated to $\pi_1$ (cf. \cite{O2012}).
In \cite{O2012}, it is also showed that $O_2$ is contained a weak endoscopic lift of the pair $\pi(\mu)$ and $\pi(\mu^3)$, where $\mu$ is the gr\"o{\ss}en-character over $\Q(\sqrt{-1})_\A^\t$ such that $L(s,\mu) = L(s,E_{32})$ for the elliptic curve $E_{32}: y^2= x^3-x$, and $\pi(\mu^i)$ indicates the irreducible automorphic cuspidal representation associated to $\mu^i$ (c.f. section 12 of \cite{JL}).
Now, the last orbit $O_3$ composed of $15$ forms is generated by the following six tuple product of distinct theta constants:  
\begin{eqnarray}
F_Z(\tau) := \th_{(0,0,0,0)} \th_{(0,0,0,1)}\th_{(0,0,1,0)}\th_{(0,0,1,1)}\th_{(0,1,1,0)}\th_{(0,1,0,0)}(\tau). \label{FZ}
\end{eqnarray}
Applying the Siegel $\Phi$-operator after twisting by $g_0$ on $F_Z$ (see (\ref{eqn:defg0g2}) for the definition of $g_0$), we obtain a nontrivial elliptic modular form of weight $3$ as follows.
With $\tau_1 \in \H_1$,
\begin{eqnarray*}
\Phi(F_Z|[g_0])(\tau_1) &=& \lim_{t \to \i} F_Z\l(\begin{bmatrix}
\sqrt{-1} t & 0 \\
0 & \tau_1
\end{bmatrix}\r), \\
&=& \th_{(0,0)}^2 \th_{(0,1)}^2 \th_{(1,0)}^2(\tau_1).
\end{eqnarray*}
Thus, $F_Z$ is not a cusp form.
By using Igusa's transformation formula, one can find  that $\Phi(F_Z|[g_0])(4\tau_1)$ belongs to $M_3^1(\G^1_0(16),\chi_{-1})$, and that any $\SL_2(\Z)$-translation of $\th_{(0,0)}^2 \th_{(0,1)}^2 \th_{(1,0)}^2$ does not have a constant term.
Thus 
\begin{eqnarray*}
\Phi(F_Z|[g_0])(4\tau_1) \in S_3^1(\G^1_0(16),\chi_{-1}).
\end{eqnarray*}
This gives Theorem \ref{main2}-(ii) with Proposition \ref{dim}.

On the one hand, the central character of $\pi(\mu^2)$ is $\chi_{-1}|*|^3$.
The conductor of $\pi(\mu^2)$ is calculated as $16$, since the conductor of $\mu^2$ is $4$ and the determinant of $\Q(\sqrt{-1})$ over $\Q$ is $4$.
On the other hand, according to William Stein's table, $S_3^1(\G^1_0(16),\chi_{-1})$ is of dimension one. 
Therefore, $S_3^1(\G^1_0(16),\chi_{-1})$ is generated by the new form $f^{new}_{\mu^2}$ of $\pi(\mu^2)$, and $\Phi(F_Z|[g_0])(4\tau_1)$ is a constant multiple of $f^{new}_{\mu^2}$.
However, since all elements $F_Z'$ of the orbit $O_3$ except $F_Z$ is a multiple of  $\th_{(1,1,1,1)}$ or $\th_{(1,0,0,1)}$, it follows that
\begin{eqnarray}
\Phi(F_Z'|[g_0])= 0 \label{eqn:FZ'} 
\end{eqnarray}
from the property $\Phi(\theta_{(*,1,*,*)}|[g_0]) = 0$.
We will show that $\wt{F}_Z$ is a Hecke eigenform outside $2$ and determine $L(s,F_Z; {\rm AE})$, using the following proposition.
This proposition is a generalization of the so-called Zharkovskaya relation.
\begin{prop}\label{prop:Zhar}
Let $\G \supset \G^2(N)$ be a congruence subgroup such that $\G_v$ contains all $l_2(t,t')$ for all $t,t' \in \Z_v^\t$ at each $v \mid N$ $($see $(\ref{eqn:defln})$ for the definition of $l_2)$.
Let $\chi$ be a character on $\G$ which is trivial on $\G(N)$, and $\xi_1, \xi_2$ be the (unitary) automorphic characters determined by $\xi_{1,v}(t) = \chi_v(l_2(t,1)), \xi_{2,v}(t) = \chi_v(l_2(1,t))$ for all $v \mid N$.
Let $\{ \pi_i \}$ be the set of irreducible automorphic representations of $\GL_2(\A)$ which have vectors in $M_k^1(\G^1_0(N^2), \xi^{-1} \w_E)$.

Suppose that $E \in M_k^2(\G,\chi)$ has a nontrivial $\Phi(E)$.
Then, 
\begin{eqnarray}
\Phi(E)(N\tau) \in M_k^1(\G^1_0(N^2), \xi_2). \label{eqn:Zcentchara}
\end{eqnarray}
Further, the followings are valid. \\
1) 
If $\wt{E}$ belongs to an unramified irreducible representation $\Pi_p$ of $\GSp_2(\Q_p)$, then $\Pi_p$ is a constituent of $|*|_p^{k-2}\xi_{1,p} \rt \pi_{i,p}$ for some $i$, and 
$$L(s,\Pi_p; {\rm spin}) = L(s -k+2 ,\pi_{i,p} \ot \xi_1)L(s,\pi_{i,p}).$$ 
2) 
If $\wt{\Phi(E)}$ belongs to $\pi_{i,p}$ for some $i$, then $\Pi_p$ is a constituent of $|*|_p^{k-2}\xi_{1,p} \rt \pi_{i,p}$
\end{prop}
\begin{proof}
Since $\Phi(E) \neq 0$ and $E \in M_k^2(\G,\chi)$, we may consider that $\Pi_p$ has a nontrivial generalized Whittaker model $F_{T_1}^{\psi_0}$ such that 
\begin{eqnarray*}
F_{T_1}^{\psi_0}((t,g)) = |t|_p^k\xi_{1,p}(t) F_{T_1}^{\psi_0}((1,g))
\end{eqnarray*}
for $t \in \Q_p^\t, g \in \SL_2(\Q_p)$.
From this and that $\Phi(E) \in M_k^1(\G^1(N))$, (\ref{eqn:Zcentchara}) follows immediately. 
1)
The Siegel $\Phi$-operator gives a linear mapping 
$$\Pi_p|_{M_Q} \to |*|_p^{k-2} \xi_1 \ot \sum_i \pi_{i,p}.$$
Therefore, $\Hom_{Q}(\Pi_p|_{Q}, |*|_p^{k-2}\xi_{1,p} \ot \pi_{i,p}) \neq 0$ for some $i$.
By the Frobenius reciprocity 2.28 of \cite{BZ}, $\Hom_{\GSp_2}(\Pi_p, |*|_p^{k-2}\xi_{1,p} \rt  \pi_{i,p}) \neq 0$. 
Thus, irreducible $\Pi_p$ is a constituent of $|*|_p^{k-2}\xi_{1,p} \rt \pi_{i,p}$, and the $L$-function of $\Pi_p$ is as above.
2) Similar to 1).
\end{proof}
\begin{Rem}
When $E$ is of weight $(k+l,k)$ with $k \ge 2$ or $k =1, l=2$ such that  $\Phi(E) \neq 0$, an argument of a theta correspondence shows that $\Phi(E) \in M_{k+l}$.
In this case, the analogous relation holds.
\end{Rem}
Since $\pi(\mu^2)_p$ for odd $p$ is an unramified principal series representation, we may write
\begin{eqnarray}
\pi(\mu^2)_p = \pi(|*|_p\xi_p,|*|_p\xi_p^{-1}\chi_{-1,p}) \label{eqn:psmu}
\end{eqnarray}
by a unitary unramified quasi-character $\xi_p$ of $\Q_p^\t$ in the sense of \cite{JL}.
Since the square of $\mu_2$ does not factor through $\Q_2^\t$, $\pi(\mu^2)_2$ is supercuspidal.
\begin{thm}\label{thm:repF_Z}
The standard extension $\wt{F}_Z$ of $F_Z$ in (\ref{FZ}) belongs to the irreducible automorphic representation $\Pi = \ot_v \Pi_v$ with 
\begin{eqnarray*}
\Pi_v =
\begin{cases}
\mbox{holomorphic discrete series of minimal $K$-type $(3,3)$} & v = \infty, \\
L(|*|_p\chi_{-1,p},\pi(\mu^2)_2) & v=2, \\
L(|*|_p\chi_{-1,p}, \chi_{-1,p} \xi_p^{-2}\rtimes |*|_p\xi_p) & \mbox{otherwise}.
\end{cases}
\end{eqnarray*}
See \cite{R-S} for the meaning of non-supercuspidal irreducible admissible representations appearing above.
In particular, 
\begin{eqnarray}
L(s,\Pi;{\rm spin}) = L(s,\mu^2) L(s-1,\mu^2). \label{eqn:LFZ}
\end{eqnarray}
\end{thm}
\begin{proof}
From Proposition \ref{dim}, it follows that $\wt{F_Z}$ is a Hecke eigen form outside $2$.
Since $F_Z|[e_6] = -F_Z$, it follows from Proposition \ref{prop:Zhar} that 
\begin{eqnarray*}
\wt{F_Z} \in |*|_p\chi_{-1} \rt \pi(\mu^2)_{p}.
\end{eqnarray*}
Let $\Pi$ be the representation generated by $\wt{F_Z}$.
Consulting the Table A.8 of \cite{R-S} on spinor $L$-functions of non-supercuspidal representations, noting that $\Pi_p$ is unramified, one can determine $\Pi_p$ for odd $p$. 
Consulting the Table A.4 of loc.cit. on the semi-simplifications of Jacquet modules with respect to $N_Q$, and noting the fact that $\wt{\Phi(F_Z)} \in \pi(\mu^2)$, one can determine $\Pi_2$, and find $L(s,\Pi_2;{\rm spin}) = 1$. 
Finally, we will show the irreducibility of $\Pi$.
It is easy to show that any $f \in \Pi$ has a nontrivial $\Phi(f)$.
Therefore, the mapping $\Pi \ni f \to f_{T_1}^{\psi_0}$ is injective.
Since $f_{T_1}^{\psi_0}=\ot_v  f_{T_1,v}^{\psi_0}$, and $f_{T_1,v}^{\psi_0} \in \Pi_v$, we have now an injection $\Pi \to \ot_v \Pi_v$.
Since each $\Pi_v$ is irreducible, $\ot_v \Pi_v$ is irreducible, and so is $\Pi$.
This completes the proof.
\end{proof}
\begin{Cor}
The standard extension $\wt{F}_Z$ is a Hecke eigen form with respect to the classical Hecke operators $T_\i(h)$ for all $h \in \Delta_p$ with $p \neq 2$.
Further, $M_3^2(\G_Z)^{KE} = \C F_Z$ is closed with respect to these operators.
\end{Cor}
\begin{proof}
Since an irreducible unramified representation $\pi_p$ of $\GSp_2(\Q_p)$ has the unique $\GSp_2(\Z_p)$-fixed vector up to multiples, $\wt{F}_Z \in \Pi_p$ is a Hecke eigen form with respect to local Hecke operator $T_p(h)$.
The assertions follow from Proposition \ref{prop:egvgl} and the one-dimensionality of $M_3^2(\G_Z)^{KE}$ showed in Proposition \ref{dim}.
\end{proof}
\begin{Rem}
Let $\G \subset \Sp_n(\Z)$ and $\chi$ be a character on $\G$.
As pointed in Lemma 3.1. of \cite{ST}, in general $T_\i(h)$ does not preserve $M_\rho(\G,\chi)$ even if $\dim_\C M_\rho(\G,\chi) = 1$.
Indeed, for example $T_\i(a_2(p))$ with $p \equiv -1 \pmod{4}$ sends $M_3(\G(2),\chi_Z)$ to $M_3(\G(2),\ol{\chi}_Z)$, where $\chi_Z$ is the character obtained from $F_Z$ by the Igusa transformation formula, and $\ol{\chi}_Z (\neq \chi_Z)$ denotes a conjugation of $\chi_Z$ (see loc. cite.).
However, $T_\i(a_2(p)) F_Z$ is $0$, and still lives in $M_3(\G(2),\chi_Z)$.
Let $M_3(\G)^{SE-}$ denote the subspace of $M_3^2(\G)$ generated by Siegel modular forms $F$ such that $\Phi(F|[\g])$ is cuspidal for any $\g \in \Sp_2(\Z)$. 
Then, $M_3(\G)^{SE-}$ is the direct sum of $M_3(\G)^{KE}$ and $S_3(\G)$.
Then, since $\dim_\C M_3(\G_Z)^{SE-} = 1$ and $\G_Z \subset {\rm Ker}(\chi_Z) = {\rm Ker}(\ol{\chi}_Z)$, 
\begin{eqnarray*}
M_3(\G_Z)^{SE-} = M_3({\rm Ker}(\chi_Z))^{SE-} = M_3(\G(2),\chi_Z)^{SE-} = \C F_Z, \ M_3(\G(2),\ol{\chi}_Z)^{SE-} = 0.
\end{eqnarray*} 
\end{Rem}
\subsection{Soudry lift.}
Let $K_v = \Q_v(\sqrt{d})$ with $d \in \Q_v^\t \setminus (\Q_v^\t)^2$ or $\Q_v^2$ for a place $v$ of $\Q$. 
Let $c$ be the generator of ${\rm Gal}(K_v/\Q_v)$.  
In case $K_v=\Q_v^2$ (resp. $K_v= \Q_v(\sqrt{d})$), we define a quadratic form by $((x,y),(x',y')) = xy' + x'y$ for $(x,y),(x',y') \in \Q_v^2$ (resp. $(z,z') = \frac{1}{2}(zz'^c + z^cz')$ for $z,z' \in \Q_v(\sqrt{d})$).
Let $\GO_{K_v}(\Q_v)$ denote the generalized orhogonal group of $K_v$, and $\nu$ the similitude character of $g \in \GO_{K_v}(\Q_v)$.
Let $\GSO_{K_v} = {\rm Ker}(\det/\nu)$.
Then, $\GSO_{K_v}(\Q_v)$ is isomorphic to $K_v^\t$, and 
\begin{eqnarray*}
\GO_{K_v}(\Q_v) \simeq K_v^\t \rtimes \Z/2\Z. 
\end{eqnarray*} 
In case of $K_v = \Q_v(\sqrt{d})$, $\Z/2\Z = \{ 1 ,c \} \simeq {\rm Gal}(K_v/\Q_v)$ acts on $K_v$, and in case of $K_v = \Q_v^2$, $\{ 1 ,c \}$ acts on $\Q_v^2$ by the permutation.
The law of group is defined by $(h,c^n)(h',c^m) =(hh'^{c^n},c^{n+m})$.

Let $\s$ be an automorphic, unitary, quasi-character on $K_\A^\t$.
Let $\In (\s_v)$ denote the induced ($2$-dimensional) representation of $\GO_K(\Q_v)$ from $\s_v$ of $\GSO_K(\Q_v)$.
If $\s_v \neq \s_v^c$, then $\In (\s_v)$ is irreducible. 
If $\s_v = \s_v^c$, then $\In(\s_v) = \s_v^+ \oplus \s_v^-$ with $\s_v^{\pm}$ irreducible, where $\s_v^+(h,c^n) = \s_v(h), \s_v^-(h,c^n) = (-1)^n\s_v(h)$.
Let $S$ be the set of places $v$ at which $\s_v$ is ramified, and
\begin{eqnarray*}
R = \{v \mid \s_v^c = \s_v\}.
\end{eqnarray*}
Although $|S|$ is finite, $|R|$ may be infinite.
Each automorphic irreducible constituent $\wh{\s}$ of the induced representation from $\s$ is in a shape of 
\begin{eqnarray*}
\l(\ot_{v \in R_+} \s_v^+ \r) \ot\l(\ot_{v \in R_-} \s_v^- \r) \ot \l(\ot_{v \not\in R} \In(\s_v)\r)
\end{eqnarray*}
where $R = R_+ \sqcup R_-$ and $|R_-| < \infty$.
If $p \in S \cup R_-$, then $\wh{\s}_p$ is ramified.  
There is a vector $f_0 = \ot_v f_{0v} \in \wh{\s}$ such that 
\begin{eqnarray}
f_{0v}(h,c^l) =
\begin{cases}
(-1)^l \s_v(h) & \mbox{ at $v \in R_-$,} \\
\s_v(h) & \mbox{otherwise.}
\end{cases} \label{eqn:newvec}
\end{eqnarray}
Let $w_v^n$ denote the Weil representation of $\Sp_n \t \O_K$.
For $\vp = \ot_v \vp_v \in \Sc(K_\A^n)$, we define the $\th$-kernel associated to $\vp$ by
\begin{eqnarray*}
\th_n(\vp)(g,u) = \sum_{z \in K^n} \ot_v w_v^n(g,u) \vp_v(z).
\end{eqnarray*}
Further, for $f = \ot_v f_v \in \wh{\s}$, we define 
\begin{eqnarray}
\th_n(\vp,f)(g) = \int_{\O_K(\Q) \bs \O_K(\A)} \th_n(\vp)(g,u) f(u) du, \label{eqn:th-lift}
\end{eqnarray}
which is an automorphic form on $\Sp_n(\A)$.
The standard extension associated $\th_n(\vp,f)$ naturally is also denoted by $\th_n(\vp,f)$.  
We denote by $\th_n(\wh{\s})$ the subspace of automorphic forms on $\GSp_n(\A)$ spanned by them.
The central character of $\th_n(\wh{\s})$ is $\s|_{\A^\t}$.
\begin{Def}[Soudy lift]
We call $\th_2(\wh{\s})$ a Soudry lift of $\s$.
\end{Def}
\begin{Rem}
$\th_2(\wh{\s})$ is cuspidal, if and only if $R_- \neq \emptyset$ (cf. Lemma 1.3 of \cite{Soudry}).
\end{Rem}
By a computation of Hecke operators as in \cite{Y}, one can find that
\begin{eqnarray}
L(s,\th_2(\wh{\s})_v;{\rm spin}) = L(s,\s_v)L(s-1,\s_v) \label{eqn:LSoudrylift}
\end{eqnarray}   
for $v \not\in S$.
Let us observe Bessel functions of $F \in \th_2(\wh{\s})$.
Let $\vp = \ot_v \vp_v \in \Sc(K_\A^2)$, $f = \ot_v f_v$ of an irreducible automorphic representation $\wh{\s}$ of $\GO_K(\A)$, and $F = \th_2(\vp,f)$.
First of all, $F_T$ is given by 
\begin{eqnarray*}
F_T(g) = \int_{\O_K(\Q) \bs \O_K(\A)} \sum_{\g \in \O_K(\Q)} w(g,u) \vp(\g^{-1}z_1,\g^{-1}z_2) f(u) du 
\end{eqnarray*}
if $T = \begin{bmatrix}
(z_1,z_1) & (z_1,z_2) \\
(z_1,z_2) & (z_2,z_2)
\end{bmatrix}$.
Therefore $-\det(T) \in N_{K/\Q} (K)$. \\
{\bf A)} Suppose that $K = \Q(\sqrt{-d_K})$ with $-d_K \in \Q^\t \setminus (\Q^\t)^2$.
Put $z_0 = \sqrt{-d_K}$, and $T = \begin{bmatrix}
1 & 0\\
0 & d_K
\end{bmatrix}$.
Then,  
\begin{eqnarray}
F_T(g) &=& \int_{\O_K(\Q) \bs \O_K(\A)} \sum_{\g \in \O_K(\Q)} w(g,\g u) \vp(1,z_0) f(u) du \nonumber \\
&=& \int_{\O_K(\A)} w(g,u) \vp(1,z_0) f(u) du = \bigotimes_v F_{T,v}(g_v), \label{eqn:Fourierdec}\\
F_{T,v}(g_v) &=& \int_{\O_K(\Q_v)} w_v(g_v,u_v) \vp_v(1,z_0) f_v(u_v) du_v \label{eqn:locFourier}
\end{eqnarray}
where we normalize the Haar measure $d u_v$ so that ${\rm vol}(\O_K(\Q) \bs \O_K(\A))=1$.
Let $\xi$ be an automorphic character on $K^1 \bs K_\A^1$.
The Bessel function associated to $\xi$ is calculated as
\begin{eqnarray}
F_T^\xi (g) &=& \int_{K^1 \bs K_\A^1}\xi(k)^{-1} \int_{\O_K(\A)} w\l(\begin{bmatrix}
k & 0_2\\
0_2 & {}^tk^{-1}  
\end{bmatrix}g,1\r) \vp(u^{-1},u^{-1}z_0) f(u) du d k  \nonumber \\
&=& 
\int_{K^1 \bs K_\A^1} \xi(k)^{-1} \int_{\O_K(\A)} w(g,1)\vp\l(\l(u^{-1},u^{-1}z_0\r)k\r) f(u) du d k \nonumber \\
&=& 
\int_{K^1 \bs K_\A^1} \xi(k)^{-1} \int_{\O_K(\A)} w(g,1)\vp\l(u^{-1}k,u^{-1}kz_0\r) f(u) du d k \label{eqn:besk}
\end{eqnarray}
where we normalize the Haar measure $d k$ so that ${\rm vol}(K^1 \bs K_\A^1)=1$.
From this equality it follows that $F_T^\xi$ vanishes if $\xi \neq (\s|_{K_\A^1})^{-1}$.
Thus, we have the equality
\[
F_T = F_T^{\s^{-1}}.
\]
{\bf B)}
Suppose that $K = \Q^2$.
In this case, $K^\t = (\Q^\t)^2, K^1 \simeq \Q^\t$, and $\s$ factors as $\s(a,b) = \s_1(a) \s_2(b)$ for characters $\s_1,\s_2$ of $\Q_v^\t$.
Let $\{e_+,e_{-}\}$ be the standard basis of $K$: $(e_+,e_-) = 1, (e_+,e_+) = (e_-,e_-) = 0$.
Put $T = \begin{bmatrix}
 0& 1 \\
1 &0 
\end{bmatrix}$, and replace the above $(1,z_0) $ with $(e_+,e_-)$.
Then, $\SO_T(\Q_v) = \Bigg\{\begin{bmatrix}
a & 0 \\
0 & a^{-1}
\end{bmatrix}\ \Bigg|\ a \in \Q_p^\t \mbox{($\R^\t$ if $v = \infty$)}\Bigg\}$, and denote by $\xi$ the character sending $\begin{bmatrix}
a & 0 \\
0 & a^{-1}
\end{bmatrix}$ to $\xi(a)$.
Then, we have
\begin{eqnarray}
F_T(g) &=& \int_{\Q^\t \bs \A^\t} \xi(a)^{-1} \int_{\Q^\t \bs \A^\t} w(g,1)\vp\l(a k^{-1}e_+, a^{-1}k e_-\r) \s_1\s_2^{-1}(a) d a d k \label{eqn:besk'} \\
& & + \int_{\Q^\t \bs \A^\t} \xi(a)^{-1} \int_{\Q^\t \bs \A^\t} w(g,1)\vp\l(a k^{-1}e_-, a^{-1}k e_+\r) \s_1\s_2^{-1}(a) d a d k. \nonumber
\end{eqnarray}
and therefore $F_T = F_T^{\s_1^{-1}\s_2}$.
Now, consider the case of $f= f_0$ as in (\ref{eqn:newvec}).
Let $\oo_v$ denote the ring of integers of $K_v$.
Let $\K$ denote the maximal compact subgroup of $\GO_K(\A)$ such that, via the isomorphism $\SO_K(\Q_v) \simeq K_v^1$, it holds
\begin{eqnarray*}
\K_v \cap \SO_K(\Q_v) \simeq 
\begin{cases}
\C^1 & \mbox{if $v = \infty$ and $K_v \simeq \C$,} \\
\{(1,1)\} \in (\R^\t)^2 & \mbox{if $v = \infty$ and $K_v \simeq \R^2$,} \\
\oo_v^1 & \mbox{if $v < \infty$ and $K_v/\Q_v$ does not split,} \\
\{(a,a^{-1}) \in (\Z_v^\t)^2 \} (\simeq \Z^\t)& \mbox{if $v < \infty$ and $K_v \simeq \Q_v^2$,} 
\end{cases}
\end{eqnarray*}
and therefore we will identify them. 

For future use and simplicity, we introduce the following notion. 
The reader should not be confused with the similar notion in ``fundamental lemma" for orbital integrals.  
\begin{Def}({\bf Matching} of Schwartz-Bruhat functions) 
Let $n$ be a positive integer.
If $\vp_v \in \Sc(K_v^n)$ satisfies the following condition, we say that $\vp_v$ matches to $\hat{\s}_v$.
For any $h \in \K_v$,
\begin{eqnarray}
w_v(1,(h,c^l))\vp_{v} = \vp_{v} \t
\begin{cases}
(-1)^l \s_v(h)^{-1} & \mbox{if $v \in R_-$,} \\
\s_v(h)^{-1} & \mbox{otherwise.}
\end{cases} \label{eqn:propvp}
\end{eqnarray}
\end{Def}
If $\vp_v$ matches to $\hat{\s}_v$, then we have
\begin{eqnarray}
\int_{\O_K(\Q_v)} w_v(g,u)\vp_v(z_1,z_2) f_0(u) du 
&=& 2\int_{\SO_K(\Q_v)} w_v(g,u)\vp_v(z_1,z_2) \s(u) d u \nonumber \\
&=& 2\int_{K_v^1/\K_v \cap \SO_K(\Q_v)} w_v(g, \dot{u}) \vp_v(z_1,z_2) \s(\dot{u}) d \dot{u}, \label{eqn:intbessel}
\end{eqnarray} 
where $\dot{u}$ the image of $u \in K_v^1$ by the projection $K_v^1 \to K_v^1/\K_v \cap \SO_K(\Q_v)$. 
We are going to give such a $\vp_v$.\\
{\bf i)}
In case that $v = \infty$ and $K_v \simeq \C$, suppose that $\s_\i(z) = (z/|z|)^\ep$ with $\ep \in \Z$.
Then, letting $X,Y$ indeterminants, 
\begin{eqnarray*}
\vp_\i(z_1,z_2) = P_\ep(Xz_1+Yz_2)\exp\l(-2\pi (|z_1|^2 + |z_2|^2)\r) \t 
\begin{cases}
{\rm Im}(z_1z_2^c) & \mbox{if $\i \in R_-$,} \\
1 & \mbox{if $\i \not\in R_-$,}
\end{cases}
\end{eqnarray*}
where 
\begin{eqnarray*}
P_\ep(z) = 
\begin{cases}
(\ol{z})^\ep & \mbox{if $\ep \ge 0$,} \\
z^{-\ep} & \mbox{if $\ep <0$.}
\end{cases}
\end{eqnarray*}
Only in case of $\ep = 0$, $R_-$ may contain $\i$ (cf. \cite{KV}), and the automorphic form given by the $\th$-lift (\ref{eqn:th-lift}) for an arbitrary $f$ and $\vp$ with $\vp_\i$ as above is a holomorphic Siegel modular form of weight $(1,1)$ if $\i \not\in R_-$(resp. $(2,2)$ if $\i \in R_-$) since $K_\i$ is a positive definite quadratic space.
In case of $\ep \in \Z_{>0}$ (resp. $\ep \in \Z_{<0}$), the automorphic form given by (\ref{eqn:th-lift}) has weight $(\ep+1,1)$ (resp. $(\ep-1,-1)$), and is holomorphic (resp. anti-holomorphic) Siegel modular form. 
In particular, only in cases of $\ep = \pm 2$ the $\th$-lifts (\ref{eqn:th-lift}) may contribute to $H^2(\G \bs \H_2,\C)$ for a sufficiently small subgroup $\G$ (cf. p. 489 of \cite{o-s}). \\
{\bf ii)}
In case that $v = \infty$ and $K_v \simeq \R^2$, suppose that $\s_\i(x,y) = (x y)^{\ep}$ with $\ep \in \Z$.
Then, writing $z_i = (x_i, y_i)$, 
\begin{eqnarray*}
\vp_\i(z_1,z_2) = Q_\ep(Xz_1+Yz_2)\exp\l(-2\pi (|x_1|^2 + |x_2|^2 + |y_1|^2+ |y_2|^2)\r) \t 
\begin{cases}
x_1y_2 +x_2y_1 & \mbox{if $\i \in R_-$,} \\
1 & \mbox{if $\i \not\in R_-$,}
\end{cases}
\end{eqnarray*}
where 
\begin{eqnarray*}
Q_\ep(z) = \exp(2\pi(x^2+y^2)) \t
\begin{cases}
\int_{-\i}^\i \exp(2\pi \sqrt{-1}yy')\exp(-2\pi(x^2+y'^2))(x+\sqrt{-1} y')^\ep dy' & \mbox{if $\ep \ge 0$,} \\
\int_{-\i}^\i \exp(2\pi \sqrt{-1}yy')\exp(-2\pi(x^2+y'^2))(x-\sqrt{-1} y')^{-\ep} dy' & \mbox{if $\ep <0$.}
\end{cases}
\end{eqnarray*}
{\bf iii)}
In case that $v = p$ and $K_p$ does not split, 
\begin{eqnarray*}
\vp_p(z_1,z_2) = 
\begin{cases}
\Ch(\oo_p;z_1)\Ch(\oo_p;z_2) & \mbox{if $p \not\in S \cup R_-$,}\\
\s(z_1)\Ch(\oo_p^\t;z_1)\Ch(\oo_p;z_2) & \mbox{if $p \in S \setminus R_-$,}\\
\phi_p(z_1z_2^c)\Ch(\oo_p^\t;z_1)\s_p(z_1) & \mbox{if $p \in R_-$,}
\end{cases}
\end{eqnarray*}
where $\Ch$ indicates the characteristic function.
Here $\phi_p^- \in \Sc(K_p)$ is defined so that $\phi_p^-(z^c) = -\phi_p^-(z)$ and ${\rm supp}(\phi_p^-) = \oo_p^\t$ (this is possible). \\
{\bf iv)}
In case that $v = p$ and $K_p \simeq \Q_p^2$, writing $z_i = (x_i,y_i)$,
\begin{eqnarray*}
\vp_p(z_1,z_2) = 
\begin{cases}
\Ch(\Z_p^2;z_1)\Ch(\Z_p^2;z_2) & \mbox{if $p \not\in S \cup R_-$,}\\
\s(z_1)\Ch((\Z_p^\t)^2;z_1)\Ch(\Z_p^2;z_2) & \mbox{if $p \in S \setminus R_-$,}\\
\phi_p(x_1y_2 + x_2y_1)\Ch((\Z_p^\t)^2;z_1)\s_p(z_1) & \mbox{if $p \in R_-$,}
\end{cases}
\end{eqnarray*}
where $\phi_p^- \in \Sc(\Q_p)$ is defined so that $\phi_p^-(-x) = -\phi_p^-(x)$ and ${\rm supp}(\phi_p^-) = \Z_p^\t$.
Under the above conditions, it is easy to check that (\ref{eqn:intbessel}) does not vanish for $g = 1, (z_1,z_2) = (1,z_0)$, and therefore $F_{T_v}(1) \neq 0$ by (\ref{eqn:locFourier}).
Thus $F_T(1) \neq 0$ by (\ref{eqn:Fourierdec}).
Summarizing, 
\begin{thm}
Let $K = \Q(\sqrt{-d})$ $($resp. $\Q^2)$.
Let $\s$ be an automorphic character of $K_\A^\t$, and $\wh{\s}$ an irreducible automorphic constituent of $\In(\s)$.
Then, for $f_0 \in \wh{\s}$ and $\vp \in \Sc(K_\A^2)$ as above, $F= \th_2(\vp,f_0)$ has a nontrivial Bessel function $F_T^{\s^{-1}}$ which coincides with $F_T$ for $T = \dg(1,d)$ $($resp. $\begin{bmatrix}
0 & 1\\
1 & 0
\end{bmatrix})$.
\end{thm}
\begin{Rem}\label{rem:nongen}
Let $V$ be a quadratic space over $\Q_v$.
For a local $\th$-lift from $\O_V(\Q_v)$ to $\Sp_n(\Q_v)$, in order to have a Whittaker model, $\dim V \ge 2n$ and $\mathcal{H}^{(n-1)} \subset V$ are needed, where $\mathcal{H}$ indicates the hyperbolic plane.
In particular, for any $\s$, the local component $\th_2(\wh{\s})_v$ of the above theta lift from $\O(2)$ to $\Sp(2)$ does not have a Whittaker model, i.e., non-generic.
According to Roberts, Schmidt \cite{R-S}, every irreducible, admissible, generic representation $\pi_v$ of $\PGSp_2(\Q_v)$ for $v < \infty$ has a $K(N)$-fixed vector for a minimal $N \in \Z_{>0}$ which is called new form of $\pi$, where $K(N)$ is the paramodular group determined by the functional equation of $L(s,\pi_v;{\rm spin})$.
But, it is possible to show that $\th_2(\wh{\s})_v$ has no such a vector.
So, it is also a problem to determine a new form of such a representation.
\end{Rem}
\begin{Rem}
Let $p$ be a prime, $K = \Q(\sqrt{-p})$, and $V(p^n) = \G^2(p^n) \bs \H_2$.
Suppose $p \equiv 3 \pmod{4}$.
Theorem 6.7 of \cite{o-s} implies that, for every gr\"o{\ss}en-character $\s$ on $K_\A^\t$, each $\th$-lift $\th_2(\wh{\s})$ which contributes to the $(2,0)$-part of $H^2(V(p),\C)$ (see p. 505 of loc.cit. for the decomposition of $H^2(V(p),\C)$) is not cuspidal.
Note that $\th_2(\wh{\s})$ is holomorphic or anti-holomorphic, and hence cannot contribute to $(1,1)$-part.
This non-cuspidality is also explained as follows.
Suppose that $\th_2(\wh{\s})$ contributes to the $(2,0)$-part of $H^2(V(p),\C)$.
Then, from the explanation of the Soudry lift, it follows that $\s_\i(z) = (z/|z|)^2$, $S \subset \{p\}$, and that $R_- \subset \{ p\}$.
But, noting $K_p/\Q_p$ is ramified, we find that  every $\vp_p \in \Sc(K_p^2)$ matching to $\s_p^-$ is not $\G^2(p)$-fixed, and therefore the local theta lift $\th_2(\s_p^-)$ has no $\G^2(p)$-fixed vector.
Hence, $R_- = \emptyset$, and $\th_2(\wh{\s})$ is not cuspidal.
However, each cuspidal $\th_2(\wh{\s})$ with $\s_\i(z) = (z/|z|)^2, \i \in R_+$ may contribute to the $(2,0)$-part of $H^2(V(p^n),\C)$ for a sufficiently large $n$.
In case of $p \equiv 1, 2 \pmod{4}$, since $K_2/\Q_2$ is ramified, any $\th_2(\wh{\s})$ has no $\GSp_2(\Z_2)$-fixed vector, and thus cannot contribute to $H^2(V(p^n),\C)$. 
\end{Rem}
\subsection{The explicit differential $2$-form on $Z$.}
Let $\s = \mu^2$.
Take the irreducible constituent $\wh{\s}$ of $\ot_v \In(\s_v)$ such that $R_- = \emptyset$.
Then, $\th_2(\wh{\s})$ is not cuspidal.
As explained before, the product 
\[
\frac{\sqrt{-1}}{2}\th_{(1,0,1,0)}\th_{(1,0,1,1)}
\]
of odd theta constants defines a Schwartz-Bruhat function $\vp_2 \in \Sc(K_2^{2})$ with $K = \Q(\sqrt{-1}) \simeq V_2$.
Set $\vp = \vp_2 \ot \bigotimes_{v \neq 2} \vp_v$ with $\vp_v$ as above. 
By Remark \ref{Rem:defG_Z}, the $\th$-kernel $\th_2(\vp)$ is $\G_Z$-fixed.
Identifying $V_2$ with $K$ via an isomorphism, we can describe $\vp_2$ as
\begin{eqnarray*}
\vp_2(z_1,z_2) = \frac{\sqrt{-1}}{2}(-1)^{x_1+y_1+x_2} \Ch(z_1; \frac{1+\sqrt{-1}}{2} + \oo_2) \Ch(z_2;\oo_2),
\end{eqnarray*}
where we write 
\[
z_1 = (\ha +x_1)+ (\ha + y_1)\sqrt{-1}, z_2 = x_2 + y_2 \sqrt{-1} \in K_2.
\]
Put $E_Z = \th_2(\vp, f_0)$.
In order to see the non-triviality of $E_Z$ (non-cuspidal if nontrivial), we observe $\Phi(E_Z) = \th_1({\rm Pr}(\vp), f_0)$, where  
\begin{eqnarray*}
{\rm Pr}(\vp)_v(z) = {\rm Pr}(\vp_v)(z) = \vp_v(z,0).
\end{eqnarray*}
Then, ${\rm Pr}(\vp_v) \in \Sc(K_v)$ satisfies the condition (\ref{eqn:propvp}).
Therefore, one can find that the classical form corresponding to $\th_1({\rm Pr}(\vp), f_0)$ is given by 
\begin{eqnarray*}
\frac{\sqrt{-1}}{2} \sum_{x +y \sqrt{-1} \in (\frac{1+\sqrt{-1}}{2} + \oo_K)} (\sqrt{-1}x + y)^2 (-1)^{\frac{(x+y)}{2}}\exp\l(\pi \sqrt{-1}(x^2+y^2) \tau_1\r).
\end{eqnarray*}
Then this coincides with $g(\frac{\tau_1}{4})$ giving the non-triviality of $E_Z$  and the claim (iii)-(b) of Theorem \ref{main2} simultaneously.
This non-cuspidal form $E_Z$ is a $\C^3$-valued function 
$
E_Z(\tau) = (h_0(\tau),h_1(\tau),h_2(\tau))
$ of weight ${\rm det}^1\otimes {\rm Sym}^2$ (of $U(2)(\C)$) in the classical sense, where 
\begin{eqnarray*}
h_i(\tau) = \sum_{z_j \in K} \l(\bigotimes_{v < \i}\vp_v(z_1, z_2)\r) \ol{z}_1^{2-i} \ol{z}_2^i \exp(2\pi \sqrt{-1} {\rm tr}\l(\tau \begin{bmatrix}
N(z_1) & {\rm Re}(z_1z_2)/2 \\
{\rm Re}(z_1z_2)/2 & N(z_2)
\end{bmatrix}\r)). 
\end{eqnarray*}
Then, the differential $2$-form $E_Z^\sharp$ on $Z$ is given by
\begin{eqnarray*}
E_Z^\sharp (\tau) = h_0(\tau) d \tau_1 \wedge \tau_2 +  h_1(\tau) d \tau_1 \wedge \tau_3 + 
h_2(\tau) d \tau_3 \wedge \tau_2, \ \ (\tau = \begin{bmatrix}
\tau_1 & \tau_2 \\
\tau_2 & \tau_3
\end{bmatrix}).
\end{eqnarray*}
By (\ref{eqn:LSoudrylift}) and the same argument as in Theorem \ref{thm:repF_Z}, one can determine each local component of the non-cuspidal representation and show that $\th_2(\wh{\s})$ is irreducible. 
Summing up, 
\begin{thm}
With notations as above, the (non-cuspidal) automorphic representation $\th_2(\wh{\s}) = \ot_v \th_2(\wh{\s})_v$ is irreducible, and
\begin{eqnarray*}
\th_2(\wh{\s})_{v} = 
\begin{cases}
w_5 (\neq \Pi_\i)  & v = \infty, \\
\Pi_ v \ot |*|_v & v \neq \infty
\end{cases}
\end{eqnarray*}
(see p. 489 of \cite{o-s} for the definition of $w_5$).
The $\C^{3}$-valued holomorphic Klingen type Eisenstein series $E_Z \in \th_2(\wh{\s})$ defines the unique (up to scalar multiples) differential $2$-form $E_Z^\sharp$ on the Siegel threefold $Z$.
\end{thm}
\begin{Rem}
From the non-genericity of $\th_2(\wh{\s})_v$ (cf. Remark \ref{rem:nongen}), one can also find that $\th_2(\wh{\s})_v = \Pi_v \ot |*|_v$ for nonarchimedean $v$ consulting the table A.1. of \cite{R-S}.
\end{Rem}
\begin{Rem}
It is natural to hope that some period of $\wt{E_Z}$ is equal to $L(s,Z^{\circ}) = L(s,\th_2(\wh{\s}))$ since $E_Z^\sharp$ is a unique differential form on $Z$. 
It is possible to show that 
\begin{eqnarray*}
Z(s, \wt{E}_Z) := \int_{\Q^\t \bs \A^\t} \wt{E}_Z(a_2(t)) |t|^{s -3/2} dt
\end{eqnarray*}
coincides with $L(s,\th_2(\wh{\s}))$ up to a scalar multiple.
\end{Rem}
\begin{Rem}Let $\G$ be a congruence subgroup of $\Sp_2(\Z)$ (may have torsion elements) and $S_\G=\G\backslash\mathbb{H}_2$ the corresponding Siegel 3-fold. 
Then one can prove that 
any Hecke eigen form $F$ in $M_{(3,1)}(\G)$ is either CAP form associated to Klingen or Klingen 
Eisenstein series as follows. As in \cite{o-s} one can prove that $H^2(S_\G,\C)$ is pure of weight 2. 
By the classification  of automorphic representation of $\GSp_2(\A)$ 
contributing to $M_{(3,1)}(\G)\simeq H^{2,0}(S_\G,\C)$ due to Weissauer \cite{wei1},\cite{wei2}, 
one can conclude the claim. We denote by  $M_3(\G)^{{\rm KE,CM}}$ the space generated by Klingen Eisenstein series whose 
image under $\Phi$ is a CM modular form. 
Then it seems to be interesting to discuss whether $M_{(3,1)}(\G)^{{\rm Eisen}}$ is naturally isomorphic to $M_3(\G)^{{\rm KE,CM}}$ 
as a Hecke module or not.  
\end{Rem}



\begin{thebibliography}{}
\bibitem{arakawa}T. Arakawa, Vector-valued Siegel's modular forms of degree two and the associated Andrianov L-functions.
 Manuscripta Math. 44 (1983), no. 1-3, 155-185.
\bibitem{b&s}S. B\"ocherer and R. Schulze-Pillot, On a theorem of Waldspurger and in Eisenstein series of 
Klingen type, Math ann. 288 (1990),no. 3, 361-388.
\bibitem{borel}A. Borel, Introduction aux groupes arithm\'etiques, 
Publications de l'Institut de Math\'ematique de l'Universit\'e de Strasbourg, XV. Actualit\'es Scientifiques et Industrielles, No. 1341 Hermann, Paris 1969, 125 pp. 
\bibitem{BZ}I.N.Bernstein and A.N.Zelevinskii, Representations of the group $GL(n,F)$ where $F$ is a nonarchimdean local field, Russian Math. Survay {\bf 31} (1976), 1-68.
\bibitem{deligne-hodge}P. Deligne, Th\'eorie de Hodge : II. Publications Math\'ematiques de l'IH\'ES, 40 (1971), p. 5-57.
\bibitem{Ev}S. A. Evdokimov, Euler products for congruence subgroups of the Siegel group of genus $2$, Math. USSR Sbornik, {\bf 28} (1976), No. 4 431-458.
\bibitem{vG-N}B. van Geemen and N. O. Nygaard, On the Geometry and Arithmetic of Some Siegel Modular Threefolds, J. Number Theory, {\bf 53} (1995), 45-87.
\bibitem{GS}B. van Geemen and D. van Straten, The cuspform of weight $3$ on $\G_2(2,4,8)$, Math. Comp., {\bf 61} (1993), 849-872.
\bibitem{ichikawa}T. Ichikawa, Vector-valued $p$-adic Siegel modular forms, J. reine angrew. Math (2012).
\bibitem{Ig}J. Igusa, On the graded ring of theta-constants. Amer. J. Math. 86 1964 219-246.
\bibitem{igusa}J. Igusa, A desingularization problem in the theory of Siegel modular functions.
Math. Ann. 168 1967 228-260. 
\bibitem{JL}H. Jacquet, R.P. Langlands, {\it Automorphic forms on GL(2),} Springer, L. N. M., {\bf 114} (1970).
\bibitem{KV}M. Kashiwara, L. Vergne, On the Segel-Shale-Weil representations and harmonic polynomials, Invent. Math., {\bf 44} (1987), 1-44.
\bibitem{klingen}H. Klingen, Introductory lectures on Siegel Modular forms, Cambridge Studies in Advanced Mathematics, 1990. 
\bibitem{manin}Ju. I. Manin, Correspondences, motifs and monoidal transformations. Math. USSR-Sb. 6 (1968), 439--470.
\bibitem{milne}J-S. Milne, \'Etale cohomology,  
Princeton Mathematical Series, 33. Princeton University Press, Princeton, N.J., 1980.
\bibitem{O2012}T. Okazaki, $L$-functions of $S_3(\G(2,4,8))$, J. Number Theory. {\bf 132} (2012), 54-78.
\bibitem{Onote}T. Oda, Real Harmonic Analysis for Geometric Automorphic Forms, preprint (available on his homepage). 
\bibitem{o-s}T. Oda and J. Schwermer, Mixed Hodge structures and automorphic forms for Siegel modular varieties of degree two. Math. Ann., 286 (1990), no. 1-3, 481--509.
\bibitem{otsubo}N. Otsubo, Selmer groups and zero-cycles on the Fermat quartic surface. J. reine angew. Math. 525 (2000), 113-146. 
\bibitem{Ro}B. Roberts: Global $L$-packets for GSp(2) and theta lifts, Docu. Math. {\bf 6} (2001) 247-314.
\bibitem{R-S}B. Roberts and R. Schmidt, Local newforms for 
GSp(4), L.N.M. 1918 (2007), Springer.
\bibitem{saito}T. Saito, Modular forms and $p$-adic Hodge theory. Invent. Math. 129 (1997), no. 3, 607-620.
\bibitem{Sa-T}P.J. Sally, Jr.,M, Tadi\'c, Induced representations and classifications for $GSp(4,F)$ and $Sp(4,F)$. M\'em. Soc. Math. France (N.S), {\bf 52}, (1993) 75-133.  
\bibitem{ST}R. Salvati Manni and J. Top, Cuspforms of weight 2 for the group $\Gamma(4,8)$, Amer. J. Math., {\bf 115} (1993), 455-486.
\bibitem{Schwermer}J. Schwermer, On arithmetic quotients of the Siegel upper half space of degree two. 
Compositio Math. 58 (1986), no. 2, 233-258.
\bibitem{shimura-mahoro}M. Shimura, Defining equations of modular curves $X_0(N)$ . Tokyo J. Math. 18 (1995), no. 2, 443-456. 
\bibitem{shimura-cm}G. Shimura, On elliptic curves with complex multiplication as factors of the Jacobians of modular function fields.
 Nagoya Math. J. 43 (1971), 199-208. 
\bibitem{Shimura}G. Shimura, Introduction to the arithmetic theory of automorphic functions. Reprint of the 1971 original. Publications of the Mathematical Society of Japan, 11. Kano Memorial Lectures, 1. Princeton University Press, Princeton, NJ, 1994.  
\bibitem{stein}W. Stein, database of modular forms, available on his homepage.  
\bibitem{Soudry}D. Soudry, The CAP representation of $GSp(4,\A)$, J. reine angrew. Math. {\bf 383} (1988), 87-108.
\bibitem{weissauer-vector}R. Weissauer, Vektorwertige Siegelsche Modulformen kleinen Gewichtes. J. reine angrew. Math. 343 (1983), 184-202.
\bibitem{wei1}R. Weissauer, Differentialformen zu Untergruppen der Siegelschen Modulgruppe zweiten Grades. J. reine angrew. Math.  391  (1988), 100-156. 
\bibitem{wei-coh}R. Weissauer, On the cohomology of siegel modular threefolds, Lecture Notes in Mathematics Volume 1399 (1989), 155-170.
\bibitem{wei2}R. Weissauer, The Picard group of Siegel modular threefolds. J. reine angrew. Math.  430  (1992), 179-211.
\bibitem{Y}H. Yoshida, Siegel's modular forms and the arithmetics of quadratic forms, Invent. Math., {\bf 60} (1980), 193-248.
\bibitem{Zhar}N. A. Zharkovskaya, The Siegel operator and Hecke operators, Functs. Anal. Appl., {\bf 8}(1974),113-120.
\end{thebibliography}
\end{document}